\documentclass{amsart}

\newif\iffinalrun

\usepackage{amsmath}
\usepackage{latexsym}
\usepackage{amsfonts}
\usepackage{amssymb}
\usepackage{graphicx}
\usepackage{mathrsfs}
\usepackage{graphpap}
\usepackage{color}

\theoremstyle{plain}

\iffinalrun
  \newcommand{\need}[1]{}
  \newcommand{\mar}[1]{}
\else
  \newcommand{\need}[1]{{\tiny *** #1}}
  \newcommand{\mar}[1]{\marginpar{\raggedright\tiny #1}}
\fi

\usepackage[all]{xy}
\newtheorem{defi}{Definition}[section]

\newtheorem{lem}[defi]{Lemma}
\newtheorem{theo}[defi]{Theorem}
\newtheorem{conj}[defi]{Conjecture}

\newtheorem{cor}[defi]{Corollary}
\newtheorem{rmk}[defi]{Remark}
\newtheorem{prop}[defi]{Proposition}

\newcommand{\defeq}{\stackrel{\textrm{\tiny{def}}}{=}}
\newcommand{\maq}[4]{\left(\begin{array}{cc} #1&#2\\#3&#4\end{array}\right)}
\newcommand{\maqn}[2]{\left(\begin{array}{#1} #2\end{array}\right)}
\numberwithin{equation}{defi}

\newcommand{\fp}{\mathbf{F}_p}

\newcommand{\N}{\mathbf{N}}
\newcommand{\Q}{\mathbf{Q}}

\newcommand{\qp}{\mathbf{Q}_p}
\newcommand{\Qp}{\mathbf{Q}_p}
\newcommand{\qpbar}{\overline{\qp}}

\newcommand{\Zp}{\mathbf{Z}_p}

\newcommand{\F}{\mathbf{F}}
\newcommand{\Oe}{\mathcal{O}_E}
\newcommand{\Z}{\mathbf{Z}}

\newcommand{\gl}[2]{\mathbf{GL}_{#1}(#2)}
\newcommand{\bT}{\mathbf{T}}

\newcommand{\teich}[1]{\widetilde{#1}}
\newcommand{\galorep}{\overline{\rho}}

\newcommand{\vectln}[2]{\left(\begin{array}{#1} #2\end{array}\right)}
\newcommand{\SBr}{S_{\mathcal{O}_E}}

\newcommand{\OEModdd}[1][r]{\text{$\Oe$-$\mathrm{Mod}_{\mathrm {dd}}^{#1}$}}
\newcommand{\OEModddN}[1][r]{\text{$\Oe$-$\mathrm{Mod}_{\mathrm {dd},0}^{#1}$}}
\newcommand{\RModdd}[1][r]{\text{$R$-$\mathrm{Mod}_{\mathrm {dd}}^{#1}$}}
\newcommand{\RModddN}[1][r]{\text{$R$-$\mathrm{Mod}_{\mathrm {dd},0}^{#1}$}}
\newcommand{\FBrModdd}[1][r]{\text{$\F$-$\operatorname{BrMod}_{\mathrm {dd}}^{#1}$}}
\newcommand{\FBrModddN}[1][r]{\text{$\F$-$\operatorname{BrMod}_{\mathrm {dd},0}^{#1}$}}
\newcommand{\RBrModdd}[1][r]{\text{$\Rbar$-$\operatorname{BrMod}_{\mathrm {dd}}^{#1}$}}
\newcommand{\RBrModddN}[1][r]{\text{$\Rbar$-$\operatorname{BrMod}_{\mathrm {dd,0}}^{#1}$}}

\newcommand{\rhobar}{\overline{\rho}}

\newcommand{\barS}{\overline{S}}
\newcommand{\cO}{\mathcal{O}}
\newcommand{\ra}{\rightarrow}
\newcommand{\Mat}{\mathrm{Mat}}
\newcommand{\phz}{\varphi}

\newcommand{\barcM}{\overline{\mathcal{M}}}
\newcommand{\barc}[1]{\overline{\mathcal{#1}}}
\newcommand{\Tst}{\mathrm{T}_{\mathrm{st}}}
\newcommand{\Tqst}{\mathrm{T}_{\mathrm{qst}}}
\newcommand{\Tcris}{\mathrm{T}_{\mathrm{cris}}}
\newcommand{\Fil}{\mathrm{Fil}}
\newcommand{\Gal}{\mathrm{Gal}}
\newcommand{\cM}{\mathcal{M}}
\newcommand{\cI}{\mathcal{I}}
\newcommand{\cJ}{\mathcal{J}}

\newcommand{\rM}{\mathrm{\mathbf{M}}}
\newcommand{\barcN}{\overline{\mathcal{N}}}
\newcommand{\rGL}{\mathrm{\mathbf{GL}}}
\newcommand{\GL}{\mathrm{\mathbf{GL}}}
\newcommand{\rU}{\mathrm{\mathbf{U}}}
\newcommand{\rB}{\mathrm{\mathbf{B}}}
\newcommand{\rT}{\mathrm{\mathbf{T}}}
\newcommand{\rL}{\mathrm{\mathbf{L}}}
\newcommand{\sR}{\mathscr{R}}
\newcommand{\opp}{\mathrm{opp}}

\newcommand{\Rbar}{\overline{R}}
\newcommand{\Fp}{\fp}
\newcommand{\onto}{\twoheadrightarrow}
\newcommand{\Frob}{\mathrm{Frob}}

\newcommand{\Spf}{\mathrm{Spf}}
\newcommand{\into}{\hookrightarrow}

\newcommand{\un}[1]{\underline{#1}}

\DeclareMathOperator{\Hom}{Hom}
\newcommand{\WD}{{\operatorname{WD}}}
\newcommand{\ilim}{\varinjlim} 
\newcommand{\Dst}{\mathrm{D}_{\mathrm{st}}}
\newcommand{\Bst}{\mathrm{B}_{\mathrm{st}}}

\author{Brandon Levin}
\address{The University of Chicago, 5734 S. University Avenue, Chicago, Illinois 60637, USA}
\email{bwlevin@math.uchicago.edu}

\author{Stefano Morra}
\address{Institut Montpelli\'erain A. Grothendieck,
Universit\'e de Montpellier,
Cc 051, Place E. Bataillon,
34095 Montpellier Cedex, France}
\email{stefano.morra@umontpellier.fr}

\keywords{}
\title{Potentially crystalline deformation rings in the ordinary case}

\begin{document}
\begin{abstract}
We study potentially crystalline deformation rings for a residual, ordinary Galois representation $\rhobar: G_{\Qp}\rightarrow \GL_3(\Fp)$. We consider deformations with Hodge-Tate weights $(0,1,2)$ and inertial type chosen to contain exactly one Fontaine-Laffaille modular weight for $\rhobar$. We show that, in this setting, the potentially crystalline deformation space is formally smooth over $\Zp$ and any potentially crystalline lift is ordinary. The proof requires an understanding of the condition imposed by the monodromy operator on Breuil modules with descent datum, in particular, that this locus mod $p$ is formally smooth.
\end{abstract}
\maketitle
\tableofcontents

\section{Introduction}

One of the aims of the $p$-adic Local Langlands correspondence is a description of $p$-adic automorphic forms in terms of Galois parameters. From this perspective, Serre weight type conjectures (cf. \cite{serre-duke}, \cite{BDJ}, \cite{florian-duke}, \cite{gee-annalen}) are the first milestone to investigate a $p$-modular correspondence (\cite{breuilI},\cite{BP}, \cite{BH}) and predict the structure of certain local deformation rings which are relevant for the refined modularity lifting techniques (\cite{kisin-annals}, \cite{gee-kisin}).

The Breuil-M\'ezard conjecture (\cite{BM}, \cite{BM11}, \cite{EG}) is intimately related to the above phenomena and gives an explicit relation between the irreducible components of the special fiber of local deformation rings and the representation theory of $\GL_n(\F_q)$.  In the case of $\GL_2$, the conjecture is known in the potentially Barsotti-Tate case by  \cite{gee-kisin} using modularity lifting techniques.  In general, the conjecture is closely related to deep modularity results (\cite{EG}, Theorem 5.5.2).  In this paper, we confirm an instance of the Breuil-M\'ezard conjecture for potentially crystalline deformation rings for $\GL_3$ with Hodge-Tate weights $(0,1,2)$.  

Let us be more precise. If $\rhobar: G_{\Qp}\rightarrow\GL_n(\F)$ is a continuous Galois representation, where $\F/\Fp$ is a finite extension (the ``field of coefficients'') one can consider the Galois deformation ring $R^{\Box,\tau,\lambda}_{\rhobar}$ parameterizing potentially semistable lifts of $\rhobar$ having constraints from $p$-adic Hodge theory  -a $p$-adic Hodge type $\lambda\in \Z^{n}$ and an inertial type $\tau:I_{\Qp}\rightarrow \GL_n(E)$, where $I_{\Qp}$ is the inertia subgroup of $G_{\Qp}$ and $E/\Qp$ is a finite extension of $\Qp$, with residue field $\F$.

If $\lambda =(\lambda_1, \ldots, \lambda_n)$ with $\lambda_i < \lambda_{i+1}$ (i.e., regular Hodge-Tate weights), then one can naturally associate a semisimple $\GL_n(\Fp)$-representation $F(\lambda,\tau)$ with $\F$ coefficients to the constraints $(\lambda,\tau)$; in particular, if $\sigma$ denotes a \emph{weight}, i.e. an isomorphism class of irreducible $\GL_n(\Fp)$-representation, we can consider the multiplicity $m_{\sigma}(\lambda,\tau)\in\N$ of $\sigma$ appearing in $F(\lambda,\tau)$. 

\begin{conj}[Breuil-M\'ezard conjecture]
There exists integers $\mu_{\sigma}(\rhobar)\in \N$ such that, for any regular Hodge type $\lambda\in \Z^n$ and any inertial type $\tau$, one has
$$
\mathrm{HS}\left(R^{\Box,\tau,\lambda}_{\rhobar}\otimes\F\right)=\sum_{\sigma} m_{\sigma}(\lambda,\tau)\mu_{\sigma}(\rhobar)
$$
where $\mathrm{HS}\left(R^{\Box,\tau,\lambda}_{\rhobar}\otimes\F\right)$ denotes the Hilbert-Samuel multiplicity of the special fiber of the deformation space $R^{\Box,\tau,\lambda}_{\rhobar}$.
\end{conj}

When the Hodge type is given by $\lambda=(0,1,\dots,n-1)$ then $F(\lambda,\tau)$ is the mod-$p$ reduction of an automorphic type, and the Serre weight conjectures can be considered as a prediction for the intrinsic multiplicities $\mu_{\sigma}(\rhobar)$.   As soon as $n\geq 3$, several complications arise in questions related to Breui-M\'ezard and modularity lifting.  On the representation theory side, the mod $p$ irreducible representations of $\GL_n(\Fp)$ no longer arise as reductions of lattices in algebraic representations over $\Zp$.  This leads to the phenomenon of \emph{shadow} weights (see \cite{florian-duke}, \cite[\S 6]{EGH}).   On the Galois side, the integral $p$-adic Hodge theory used to study potentially crystalline deformation rings of type $\lambda$ becomes more subtle.     

More precisely, let $K/\Qp$ finite extension (which we assume to be totally ramified for simplicity) and let $\pi$ be a uniformizer of $K$.  Choose a compatible system of $p$-power roots $\pi^{1/p^n}$ and define $K_{\infty} = \cup_n K(\pi^{1/p^n}) \subset \overline{K}$.  Let $G_{K}$ denote the absolute Galois group of $K$ and let $G_{K_{\infty}} := \Gal(\overline{K}/K_{\infty})$.   In general, the restriction functor 
$$
\mathrm{Rep}_{G_K}^{\mathrm{cris}} (\overline{\Qp}) \ra \mathrm{Rep}_{G_{K_{\infty}}} (\overline{\Qp}) 
$$
is fully faithful and its image is contained in the  \emph{finite height} $G_{K_{\infty}}$-representations (\cite[Corollary 2.1.14]{kisin-crystal}).  Both these categories are described by linear algebra data using $p$-adic Hodge theory.  From that perspective, the essential image of the functor is characterized by a Griffiths transversality condition.   In the Barsotti-Tate case (i.e., height $\leq 1$), Griffiths transversality is always satisfied and this gives more precise control over Barsotti-Tate deformation rings.  The difficulty that arises in higher weight situations is to understand which representations of $G_{K_{\infty}}$ descend to (potentially) crystalline representations of $G_K$ (integrally as well).  In this paper, we address this question for tamely potentially crystalline deformation rings for $\Qp$ and $\GL_3$ with Hodge-Tate weights $(0,1,2)$ with some assumptions on $\rhobar$.  

In order to state the main theorem, we let $\rhobar: G_{\Qp}\rightarrow \GL_3(\F)$ be an ordinary three dimensional Galois representation of the form
\begin{equation}
\label{Galois representation introduction}
\rhobar\vert_{I_{\Qp}}\cong \begin{pmatrix}\omega^{a_2+2}&\ast&\ast\\
0&\omega^{a_{1}+1}&\ast\\
0&0&\omega^{a_0}
\end{pmatrix}
\end{equation}
where $\omega: I_{\Qp}\rightarrow \Fp$ denotes the mod $p$ cyclotomic character and $a_i\in\N$.
Recall that $\teich{\omega}:I_{\Qp}\rightarrow \Zp^{\times}$ denotes the Teichm\"uller lift of $\omega$.

\begin{theo}[Theorem \ref{main}]
\label{theo1}
Let $\rhobar: G_{\Qp}\rightarrow \GL_3(\F)$ be an ordinary Galois representation as in $($\ref{Galois representation introduction}$)$.
Assume that the integers $a_i\in\N$ verify
$a_1-a_0,\,a_2-a_1>3$ and $a_2-a_0<p-4$ and define the inertial type $\tau\defeq \teich{\omega}^{a_2}\oplus\teich{\omega}^{a_1}\oplus\teich{\omega}^{a_0}$.

Let $R^{\Box, (0,1,2), \mathrm{cris}, \tau}_{\rhobar}$ be the framed potentially crystalline deformation ring for $\rhobar$, with Galois type $\tau$ and Hodge type $(0,1,2)$ and assume that $\Spf R^{\Box, (0,1,2), \mathrm{cris}, \tau}_{\rhobar}$ is non-empty.  Then  $R^{\Box, (0,1,2), \mathrm{cris}, \tau}_{\rhobar}$ is formally smooth of relative dimension 12.  
\end{theo}  

Theorem \ref{theo1} is obtained by explicitly constructing a formally smooth morphism $R^{\Box,\tau}_{\barcM}\rightarrow R^{\Box, (0,1,2), \mathrm{cris}, \tau}_{\rhobar}$, where $R^{\Box,\tau}_{\barcM}$ is a moduli space of strongly divisible modules $\cM$ lifting $\rhobar$ which we can control by means of integral $p$-adic Hodge theory.  There are two key ingredients.  First, for our choice of $\tau$, a detailed study of the filtration and Frobenius building on techniques of \cite{breuil-buzzati} shows that any strongly divisible module lifting $\rhobar$ is ordinary.  This gives us a formally smooth family of ordinary quasi-Breuil modules (i.e., with no monodromy operator).   Secondly, as a consequence of the genericity assumptions on $\tau$, the condition imposed by the existence of monodromy on an ordinary rank 3 Breuil module mod $p$ turns out to be exceedingly simple.   In this case, the vanishing of a single variable of the smooth family of quasi-Breuil modules.  

We now briefly discuss how Theorem \ref{theo1} is predicted by the Breuil-M\'ezard conjecture.  Recall that isomorphism classes of regular Serre weights are in bijection with triples $(a_2,a_1,a_0)\in\Z^3$ satisfying $p-1 > a_2-a_1,a_1-a_0 \geq 0$ and $p-1 > a_0 \geq 0$.  In the hypotheses of theorem \ref{main}, the inertial type $\tau$ contains exactly one weight $\sigma(a_2,a_1,a_0)$ in the conjectural set of Serre weights for $\rhobar$; it is an obvious weight for $\rhobar$ in the terminology of \cite{GHS}, in the Fontaine-Laffaille range. In particular, the Breuil-M\'ezard conjecture then predicts that $R^{\Box, (0,1,2), \mathrm{cris}, \tau}_{\rhobar}$ should be formally smooth and so Theorem \ref{theo1} confirms an instance of the conjecture for $\GL_3$.  

\begin{rmk} While proving Theorem \ref{theo1}, we also explicitly exhibit the geometric Breuil-M\'ezard conjecture of \cite{EG} in this setting.  Namely, we show that special fiber of   $\Spf R^{\Box, (0,1,2), \mathrm{cris}, \tau}_{\rhobar}$ inside the unrestricted universal framed Galois deformation space coincides with the special fiber of the $($Fontaine-Laffaille$)$ crystalline deformation ring with Hodge-Tate weights $(a_2 + 2, a_1 + 1, a_0)$. 
\end{rmk} 

As a consequence of our careful study of the filtration and Frobenius on strongly divisible modules, we get the following nice corollary:

\begin{theo}
\label{theo2}
Let $\rhobar: G_{\Qp}\rightarrow \GL_3(\F)$ be an ordinary Galois representation as in $($\ref{Galois representation introduction}$)$. Assume that the inertial type $\tau$ is as in Theorem \ref{main} and that the integers $a_i$ verify $p-4 > a_2-a_1,a_1-a_0>3$. Then, any  potentially crystalline lift $\rho$ of $\rhobar$, with Hodge type $(0,1,2)$ and inertial type $\tau$ is ordinary.
\end{theo}

We remark that for Theorem \ref{theo2}, we do not require a Fontaine-Laffaille condition on the inertial weights.

\begin{rmk} Work in progress of the two authors and Bao V. Le Hung and Daniel Le will use Kisin modules with descent datum to study potentially crystalline deformation rings for $\GL_3$ for more general $\rhobar$ (for example, semi-simple) where one expects the deformation ring not to be formally smooth.  This will have applications to Serre weight conjectures, Breuil-M\'ezard and modularity lifting. 
\end{rmk}  

The paper is organized as follows.

In \S 2, we recall various categories of semilinear algebra objects with descent data: Breuil modules, strongly divisible lattices, \'etale $\phz$-modules. We elucidate the relations among such categories and with the categories of Galois representations. We work in families, i.e. allowing coefficients in local, complete, Noetherian $\cO_E$-algebras.

The technical heart of the paper is in \S 3. After proving the uniqueness of framed Breuil module $\barcM$ associated to $\rhobar$, we perform a $p$-adic convergence argument which provides us with a complete description of the filtration and Frobenius on strongly divisible lattices lifting $\rhobar$.

In \S 4, we study the monodromy operator on Breuil modules associated to $\rhobar$. The main result, Proposition \ref{essential image}, provides us with explicit equations for the space of Breuil modules sitting inside the space of finite height modules.

The main results of the paper are in \S 5. We employ the techniques of \cite[\S 7]{EGS} to study the moduli space $R^{\Box,\tau}_{\barcM}$ of framed strongly divisible modules lifting $\barcM$. From the formal smoothness of the special fiber we deduce the main result on the formal smoothness of  the potentially crystalline deformation ring over $\Zp$ (Theorem \ref{main}). For this, we use the technique of \cite{BM11} to compare the deformation space of strongly divisible lattices and the potentially crystalline deformation ring (\ref{isomorphism moduli spaces}).

\subsection{Notations}
\label{notations}
We write $\varepsilon_p$ for the $p$-adic cyclotomic character and $\omega$ for its mod $p$ reduction. We normalize the Hodge-Tate weights in such a way that $\varepsilon_p$ has a Hodge--Tate weight $-1$. 

We consider the tamely, totally ramified extension $K/\Qp$ defined by $K\defeq \qp(\sqrt[e]{-p})$ where $e = p-1$. Recall that the choice of a uniformizer $\pi\in K$ provides us with a character 
\begin{eqnarray*}
\teich{\omega}_{\pi}: \Gal(K/\Qp)&\rightarrow&\Zp^{\times}\\
\sigma&\mapsto&\frac{\sigma(\pi)}{\pi}
\end{eqnarray*}
which induces an isomorphism $\omega_\pi:\Gal(K/\Qp)\stackrel{\sim}{\rightarrow}\Fp^{\times}$ after reduction modulo $p$. If no confusion is possible, we will simply write $\omega$ instead of $\omega_\pi$.

We fix a finite extension $E/\qp$ such that $\Hom(K,E)=\Hom(K,\qpbar)$. We write $\cO_E$ for its ring of integers,  $\F$ its residue field and $\varpi=\varpi_E\in \cO_E$ to denote an uniformizer. If $x\in \Fp$, we write $\teich{x}$ to denote its Teichm\"uller lift and, conversely, the mod $p$ reduction of an element $x\in \Zp$ will be denoted by $\overline{x}$.

We fix an embedding $K\into E$. Nothing in what follows depends on this choice. We usually write $R$ (resp. $\Rbar$) to denote a local, complete noetherian $\Oe$-algebra (resp. local artinian $\F$-algebra). If $R$ is such an algebra, we write $\Rbar$ to denote its special fiber $R\otimes_{\Oe}\F$.
All the representations and modules considered in this paper will be realized over one of the above rings $E,\, \cO_E,\, R.$
 
Given a potentially semistable $p$-adic representation $\rho: G_{\Qp}\rightarrow \GL_n(E)$, we write $\WD(\rho)$ to denote the associated Weil-Deligne representation as defined in  \cite{CDT}, Appendix B.1. We refer to $\WD(\rho)|_{I_{\Qp}}$ as the \emph{inertial type} associated to $\rho$. Note that, in particular, $\WD(\rho)$ is defined via the \emph{covariant} Dieudonn\'e module $\Dst(\rho)\defeq \underset{H/\Qp}{\ilim} (\Bst\otimes_{\Qp}\rho)^{G_{H}}$.

\section{Integral $p$-adic Hodge Theory}
\label{integralp-adicHT}

The aim of this section is to recall and extend a comparison result between Fontaine-Laffaille modules and Breuil modules with coefficients.

We write $S_{\Zp}$ to denote the usual Breuil ring: the $p$-adic completion of the divided power envelope of $\Zp[u]$ with respect to the ideal generated by the Eisenstein polynomial $E(u)\defeq u^e+p$ (compatibly with the divided power on the ideal $p\Zp[u]$). We write $\barS_{\Fp}\defeq S_{\Zp}/\left(p,\Fil^p S_{\Zp}\right)$, recalling that $\barS_{\Fp}\cong \Fp[u]/(u^{ep})$. 

If $R$ (resp. $\Rbar$) is a local noetherian $\cO_E$-algebra (resp. local artinian $\F$-algebra), we write $S_{R}$ (resp. $S_{\Rbar}$) to denote the $\mathfrak{m}_R$-adic completion of the ring $S_{\Zp}\otimes_{\Zp}R$ (resp. the ring $\barS_{\Fp}\otimes_{\Fp}\Rbar$). Note that $\barS_{\Fp}\otimes_{\Fp}\Rbar\cong \Rbar[u]/(u^{ep})$. If the rings $R$, $\Rbar$ are clear from the context, we simply write $S$, $\barS$.

The rings $S_{\Zp}$, $\barS_{\Fp}$ are endowed with additional structures. Namely, we have a continuous, semilinear Frobenius $\varphi$ (defined by $\phz(u)=u^p$), a monodromy operator $N = -u \frac{d}{du}$ and a continuous semilinear action of $\Delta\defeq\Gal(K/\qp)$ (defined by $\widehat{g}\cdot u\defeq\omega(g)u$). By base change, we obtain the evident additional structures (Frobenius, monodromy and $\Delta$ action) on $S_R$, $\barS_{\Rbar}$, endowing $R$, $\Rbar$ with the trivial Frobenius, monodromy and $\Delta$-action.

We now introduce the various categories of modules (Breuil modules, strongly divisible modules, \'etale $\varphi$-modules) and their relation to Galois representations.

A \emph{Breuil module over} $\Rbar$ is the datum of a quadruple $(\barcM_{\Rbar}, \Fil^r\barcM_{\Rbar},\varphi_r, N)$ where
\begin{enumerate}
	\item $\barcM\defeq \barcM_{\Rbar}$ is a finitely generated, free $\barS_{\Rbar}$-module; 
	\item $\Fil^r\barcM$ is a $\barS_{\Rbar}$-submodule of $\barcM$, verifying $u^{er}\barcM\subseteq \Fil^r\barcM$;
	\item the morphism $\varphi_{r}:\Fil^r\barcM\to\barcM$ is $\varphi$-semilinear and the associated fibered product $\barS_{\Fp}\otimes_{\Fp}\Fil^r\barcM\rightarrow \barcM$ is surjective;
	\item the operator $N:\barcM\to \barcM$ is $\Rbar$-linear and satisfies the following properties: 
  	\begin{enumerate}
	\item $N(P(u)x)=P(u)N(x)+N(P(u))x$ for all $x\in\barcM$, $P(u)\in\barS_{\Rbar}$; 
  	\item $u^{e}N(\Fil^r\barcM)\subseteq\Fil^r\barcM$;
  	\item $\varphi_{r}(u^{e}N(x))=N(\varphi_{r}(x))$ for all $x\in\Fil^r\barcM$.
	\end{enumerate}
\end{enumerate}	
A morphism of Breuil modules is defined as an $\barS_{\Rbar}$-linear morphism which is compatible, in the evident sense, with the additional structures (monodromy, Frobenius, filtration). If $\Rbar$ is clear from the context, we simply write $\barcM$ instead of $\barcM_{\Rbar}$.

\vspace{2mm}

A descent data relative to $\qp$ on a Breuil module $\barcM$ is the datum of an action of $\Delta$ on $\barcM$ by semilinear automorphisms and which are compatible, in the evident sense, with the additional structures on $\barcM$. 
We write $\RBrModdd[r]$ to denote the category of Breuil modules with descent data and $\Rbar$ coefficients. 

We recall (\cite{HM}, \S 2.2.1) that $\RBrModdd[r]$ is an exact category and we have an exact, faithful, contravariant functor
\begin{eqnarray*}
\Tst^*: \RBrModdd &\rightarrow&\mathrm{Rep}_{\Rbar}(G_{\Qp})\\
\barcM_{\Rbar}&\mapsto&\Tst(\barcM)\defeq \mathrm{Hom}(\barcM_{\Rbar},\widehat{A})
\end{eqnarray*}
where $\widehat{A}$ is the period ring defined in  \cite[\S 3.2]{EGH} based on \cite[\S 2.2]{breuil-inventiones}. 

We define, in the evident analogous way, the category $\RBrModddN[r]$  of qua\-si-Breuil module with descent data and coefficients: the objects and morphisms are defined as for $\RBrModdd[r]$, the only difference being that we do not require $\barcM$ to be endowed with a monodromy operator.

Recall that for a sequence $(p_n)_n\in\left(\overline{\Q}_p\right)^{\N}$ verifying $p_{n}^p=p_{n-1}$  for all $n$ we define the Kummer extension $(\Qp)_{\infty}\defeq \cup_{n\in\N}\Qp(p_n)$.
We have a faithful functor $\Tqst^*: \RBrModdd \rightarrow\mathrm{Rep}_{\Rbar}(G_{\left(\Qp\right)_{\infty}})$ by replacing $\widehat{A}$ with $A_{\text{cris}}$. The functors $\Tst^*$, $\Tqst^*$ verify the obvious compatibilities with respect to the restriction and forgetful functors: $\Tqst^*(\barcM)=\Tst^*(\barcM)\vert_{G_{\left(\Qp\right)_{\infty}}}$ if $\barcM\in \RBrModdd[r]$.

A \emph{Fontaine-Laffaille module} $(M_{\Rbar},\Fil^{\bullet}M_{\Rbar},\phi_{\bullet})$ over $\Rbar$ is the datum of 
\begin{enumerate}
	\item a finite free $\Rbar$-module $M=M_{\Rbar}$;
	\item a separated, exhaustive and decreasing filtration $\{\Fil^{j}M\}_{j\in\Z}$ on $M$ by $\Rbar$ submodules which are direct summands (the \emph{Hodge filtration});
	\item a linear Frobenius isomorphism $\phi_{\bullet}: \mathrm{gr}^{\bullet}M\rightarrow M$
\end{enumerate}

Defining the morphisms in the obvious way, we obtain the abelian category $\Rbar\text{-}\mathcal{FL}$ of Fontaine-Laffaille modules over $\Rbar$. 
Given a Fontaine-Laffaille module $M$ as above, the set of its Hodge-Tate weights is defined as
\begin{eqnarray*}
\mathrm{HT}\defeq\bigg\{i\in\N,\,\,\mathrm{rk}_{\Rbar}\bigg(\frac{\Fil^iM}{\Fil^{i+1}M}\bigg)\neq 0\bigg\}.
\end{eqnarray*}

\begin{defi}
Let $M$ be a Fontaine-Laffaille module over $\Rbar$. An $\Rbar$-basis $\underline{f}=(f_1,\dots,f_n)$ on $M$ is \emph{compatible with the filtration} if for all $i\in\N$ there exists $j_i\in\N$ such that
$\Fil^iM=\sum_{j=j_i}^{n}\Rbar\cdot f_j$. In particular, the principal symbols $(\mathrm{gr}(f_1),\dots,\mathrm{gr}(f_n))$ provide an $\Rbar$-basis for $\mathrm{gr}^{\bullet}M$.
\end{defi}

Given a Fontaine-Laffaille module and a compatible basis $\underline{f}$, it is convenient to describe the Frobenius action via a matrix $\mathrm{Mat}_{\underline{f}}(\phi_{\bullet})\in\mathrm{GL}_n(\Rbar)$, defined in the obvious way using the principal symbols $(\mathrm{gr}(f_1),\dots,\mathrm{gr}(f_n))$ as a basis on $\mathrm{gr}^{\bullet}M$.

It is customary to write $\Rbar\text{-}\mathcal{FL}^{[0,p-2]}$ to denote the full subcategory of $\Rbar\text{-}\mathcal{FL}$ formed by those modules $M$ verifying $\Fil^0M=M$ and $\Fil^{p-1}M=0$ (it is again an abelian category). We have the following description of mod $p$ Galois representations of $G_{\qp}$ via Fontaine-Laffaille modules:

\begin{theo}
\label{MainFL}
There is an exact, fully faithful contravariant functor
$$
\mathrm{T}_{\mathrm{cris}}^*:\,\,\Rbar\text{-}\mathcal{FL}^{[0,p-2]}\rightarrow\mathrm{Rep}_{\Rbar}(G_{\qp})
$$
\end{theo}

We finally recall the categories of \'etale $\varphi$-modules over $\Rbar(\!(\underline{\pi})\!)$ introduced by Fontaine (\cite{fontaine-fest}). Let $\Fp(\!(\underline{p})\!)$ be the field of norms associated to $(\Qp,p)$. In particular, $\underline{p}$ is identified with a sequence $(p_n)_n\in\left(\overline{\Q}_p\right)^{\N}$ verifying $p_{n}^p=p_{n-1}$  for all $n$. We define the category $\Rbar\mbox{-}\mathfrak{Mod}$ whose objects are free $\Rbar\otimes_{\Fp}\Fp(\!(\underline{p})\!)$-modules of finite rank $\mathfrak{D}$ endowed with a semilinear Frobenius map $\varphi:\mathfrak{D}\rightarrow \mathfrak{D}$ whose action is \'etale.

A formal modification (allowing $\Rbar$-coefficients) of work of Fontaine \cite{fontaine-fest} provides an anti-equivalence
\begin{eqnarray*}
\Rbar\mbox{-}\mathfrak{Mod}& \stackrel{\sim}{\longrightarrow}&\mathrm{Rep}_{\Rbar}(G_{(\Qp)_{\infty}})\\
\mathfrak{D}&\longmapsto& \mathrm{Hom}_{\varphi}\left(\mathfrak{D},\Fp(\!(\underline{p})\!)^{\text{sep}}\right).
\end{eqnarray*}

Let us consider $\pi\defeq \sqrt[e]{-p}\in K$. We can fix a sequence $(\pi_n)_n\in\left(\overline{\Q}_p\right)^{\N}$  such that $\pi_n^e=p_n$ for all $n\in \N$ and which is compatible with the norm maps $K(\pi_{n+1})\rightarrow K(\pi_n)$ (cf. \cite{breuil-buzzati}, Appendix A).
Letting $K_{\infty}\defeq \cup_{n\in\N}K(\pi_n)$, we have a canonical isomorphism $\Gal(K_{\infty}/(\Qp)_{\infty})\rightarrow \Delta$ and we identify $\omega$ to a character on $\Gal(K_{\infty}/(\Qp)_{\infty})$.

The field of norms $\Fp(\!(\underline{\pi})\!)$ associated to $(K, \pi)$ is then endowed with an action of $\Delta$ given by $\widehat{g}\cdot \underline{\pi}=\omega(g) \underline{\pi}$. We can therefore define the category $\Rbar\mbox{-}\mathfrak{Mod}_{\mathrm{dd}}$ of \'etale $(\varphi,\Rbar\otimes_{\Fp}\Fp(\!(\underline{\pi})\!))$-modules with descent data: an object $\mathfrak{D}$ is defined in the analogous, evident way as for the category  $\Rbar\mbox{-}\mathfrak{Mod}$, but we moreover require that $\mathfrak{D}$ is endowed with a semilinear action of $\Gal(K_{\infty}/(\Qp)_{\infty})$ and the Frobenius $\varphi$ is $\Gal(K_{\infty}/(\Qp)_{\infty})$-equivariant.

By allowing $\Rbar$-coefficients we deduce from \cite{HM}, Appendix A (building on the classical result of Fontaine) the anti-equivalence
\begin{eqnarray*}
\Rbar\mbox{-}\mathfrak{Mod}_{\mathrm{dd}}&\stackrel{\sim}\longrightarrow&
\mathrm{Rep}_{\Rbar}(G_{(\Qp)_{\infty}})\\
\mathfrak{D}&\mapsto&\mathrm{Hom}_{\varphi}\left(\mathfrak{D},\Fp(\!(\underline{\pi})\!)^{\text{sep}}\right).
\end{eqnarray*}

The main result concerning the relations between the various categories and functors introduced so far is the following:
\begin{prop}
\label{proposition comparison}
There exist faithful functors 
$$M_{\Fp(\!(\underline{\pi})\!)}: \RBrModdd[r]\rightarrow \Rbar\mbox{-}\mathfrak{Mod}_{\mathrm{dd}}$$ and
$$\mathcal{F}:\Rbar\mbox{-}\mathcal{FL}^{[0,p-2]}\rightarrow\Rbar\mbox{-}\mathfrak{Mod}$$
fitting in the following commutative diagram:
\begin{eqnarray*}
\xymatrix@=3pc{
\RBrModdd[r]\ar^{M_{\Fp(\!(\underline{\pi})\!)}}[rr]\ar^{\Tst^*}[d]&&
\Rbar\text{-}\mathfrak{Mod}_{dd}
\ar_{\hspace{-1cm}\mathrm{Hom}(\_,\Fp(\!(\underline{p})\!)^s)}[dl]\\
\mathrm{Rep}_{\Rbar}(G_{\qp})\ar^{\mathrm{Res}}[r]&\mathrm{Rep}_{\Rbar}(G_{(\qp)_{\infty}})&\\
\Rbar\text{-}\mathcal{FL}^{[0,p-2]}\ar_{\mathrm{T}_{\mathrm{cris}}^*}[u]\ar^{\mathcal{F}}[rr]&&
\Rbar\text{-}\mathfrak{Mod}
\ar^{\hspace{-1cm}\mathrm{Hom}(\_,\Fp(\!(\underline{p})\!)^s)}[ul]
\ar_{\_\otimes_{\Fp(\!(\underline{p})\!)}\Fp(\!(\underline{\pi})\!)}[uu]
}
\end{eqnarray*}
where the the functor $\mathrm{Res}\circ \mathrm{T}_{\mathrm{cris}}^*$
is fully faithful.
\end{prop}
The functors $M_{\Fp(\!(\underline{\pi})\!)}$, $\mathcal{F}$ are defined in \cite{HM}, Appendix A, building on the classical work of Breuil \cite{breuil-normes} and Caruso-Liu \cite{caruso-liu}.

In certain cases, the description of the functor $M_{\Fp(\!(\underline{\pi})\!)}$ is particularly concrete.
Assume that the Breuil module $\barcM_{\Rbar}$ has rank $n$, with descent data associated to a niveau one Galois type
${\tau}:I_{\qp}\rightarrow \mathrm{GL}_n(\cO_E)$.

By fixing a framing $\overline{\tau}=\omega^{a_1}\oplus\dots \omega^{a_n}$ we have a basis $(e_1,\dots,e_n)$ for $\barcM_{\Rbar}$  and a system of generators $(f_1,\dots,f_n)$ for $\Fil^r\barcM_{\Rbar}$ which are compatible with  $\overline{\tau}$:
$$
\widehat{g}\cdot e_i=(\omega^{a_i}(g)\otimes 1)e_i,\qquad
\widehat{g}\cdot f_i=(\omega^{a_i}(g)\otimes 1)f_i
$$
for all $i=1,\dots,n$ and all $g\in \Delta$ (cf. \cite{HM}, \S 2.2.3).

In this case we say that $\underline{e}$, $\underline{f}$ are \emph{compatible with the Galois type} $\overline{\tau}$, or that $\un{e},\,\un{f}$ are a framed basis and a framed system of generators respectively (in the terminology of \cite{HM} one would say that $\barcM_{\Rbar}$ is of type $\tau$, cf. Definition 2.2.6 in \emph{loc. cit.}).

\begin{lem}
\label{lemma lawdd 1}
Let $\barcM$ be a Breuil module of rank $n$ over $\Rbar$, with descent data associated to a Galois type
$\overline{\tau}:I_{\qp}\rightarrow \mathrm{GL}_n(\cO_E)$ and let $\underline{e}$, $\underline{f}$ be a basis for $\barcM$ and a system of generators for $\Fil^r\barcM$ respectively, which are moreover compatible with $\overline{\tau}$.

Write $V=V_{\underline{e},\underline{f}}\in M_{n}(\Rbar\otimes_{\F}\barS)$ for the matrix giving the coordinates of $\underline{f}$ in the basis $\underline{e}$ and $A\defeq \mathrm{Mat}_{\underline{e},\underline{f}}(\varphi_r)\in
\mathrm{GL}_{n}(\Rbar\otimes_{\F}\barS)$
for the matrix describing the Frobenius action on $\barcM$ with respect to $\underline{e}$, $\underline{f}$.

Then there exists a basis $\mathfrak{e}$ for $M_{\Fp(\!(\underline{\pi})\!)}(\barcM^{\ast})$ (where $\barcM^*$ denotes the associated dual Breuil module, cf. \cite{EGH}, discussion before Corollary 3.2.9), compatible with the dual descent data, such that the Frobenius action is described by 
$$
\mathrm{Mat}_{\underline{\mathfrak{e}}}(\phi)=\widehat{V}^t\big(\widehat{A}^{-1}\big)^{t}\in 
M_{n}(\Rbar\otimes_{\fp}\fp[[\underline{\pi}]])
$$
where $\widehat{V}$, $\widehat{A}$ are lifts of $V,\,A$ in $M_{n}(\Rbar\otimes_{\fp}\fp[[\underline{\pi}]])$ via the reduction morphism $\Rbar\otimes_{\Fp}\fp[[\underline{\pi}]]\onto \Rbar\otimes_{\fp}\barS$.
\end{lem}

We now recall some result in characteristic zero. Fix a positive integer $r< p-1$ and let $R$ be a complete local noetherian $\cO_E$-algebra. The category $\RModdd[r]$ of \emph{strongly divisible $R$-modules} (in Hodge-Tate weights $[0,r]$, with descent data) consists of finitely generated free $S_R$-modules ${\cM}$ together with a sub $S_R$-module $\Fil^r{\cM}$, additive maps $\varphi_r:\Fil^r{\cM}\rightarrow{\cM}$, $N:\cM\rightarrow \cM$ and $S_R$-semilinear bijections $\widehat{g}:{\cM}\rightarrow{\cM}$ for each $g\in\Delta$ such that the following conditions hold:
\begin{itemize}
\item $\Fil^r{\cM}$ contains $(\Fil^rS_R){\cM}$;
\item $\Fil^r{\cM}\cap I{\cM}=I\Fil^r{\cM}$ for all ideals $I$ of $R$;
\item $\varphi_r(sx)=\varphi(s)\varphi(x)$ for all $s\in S_R$ and $x\in{\cM}$;
\item $\varphi_r(\Fil^r{\cM})$ generates $\cM$ over $S_R$;
\item $N(sx)=N(s)x+sN(x)$ for all $s\in S_R$ and $x\in{\cM}$;
\item $N\varphi_r=p\varphi_r N$;
\item $E(u)N(\Fil^r{\cM})\subset\Fil^r{\cM}$;
\item for all $g\in\Delta$, $\widehat{g}$ commutes with $\varphi_r$ and $N$, and preserves $\Fil^r{\cM}$;
\item $\widehat{g_1\circ g_2}=\widehat{g}_1\circ\widehat{g}_2$ for all $g_1,g_2\in\Delta$.
\end{itemize}
The morphisms are $S$-module homomorphisms that preserve $\Fil^{r}{\cM}$ and commute with $\varphi_{r}$, $N$, and the descent data action.

We have a contravariant functor (cf. \cite{savitt-CDT}, \S 4 and Corollary 4.12)
$\Tst^{*}: \RModdd[r]\rightarrow\mathrm{Rep}_{R}(G_{\Qp})$ which is compatible with reduction mod $p$:
$$
\Tst^{*}(\cM)\otimes_{R}\F\cong\Tst^*(\cM\otimes_{R}\F).
$$

Let $\mathrm{Rep}^{\text{$K$-st}, [-r,0]}_{\cO_E}(G_{\Qp})$ be the category of $G_{\Qp}$-stable $\cO_E$-lattices inside $E$-valued, finite dimensional $p$-adic Galois representation of $G_{\Qp}$ becoming {semi-stable} over $K$ and with Hodge--Tate weights in $[-r,0]$.
We have a contravariant functor
$\Tst^{*}: \OEModdd[r]\rightarrow
\mathrm{Rep}^{\text{$K$-st}, [-r,0]}_{\cO_E}(G_{\Qp})$ where $\mathrm{Rep}^{\text{$K$-st}, [-r,0]}_{\cO_E}(G_{\Qp})$ is the category of $G_{\Qp}$-stable $\cO_E$-lattices inside $E$-valued, finite dimensional $p$-adic Galois representation of $G_{\Qp}$ becoming {semi-stable} over $K$ and with Hodge--Tate weights in $\{-r,0\}$ (cf. \cite{EGH}, Section 3.1, where the functor would be noted by $\Tst^{\Qp}$)

The following deep theorem provides the link between lattices in potentially semi-stable Galois representations and strongly divisible modules over $\cO_E$:

\begin{theo}[\cite{liu-breuil}, \cite{EGH}]
\label{liu}
The contravariant functor
$$
\Tst^*:\OEModdd\rightarrow \mathrm{Rep}^{\text{$K$-st}, [-r,0]}_{\Oe}(G_{\Qp})
$$
establishes an anti-equivalence of categories if $r<p-1$.

Moreover, by letting $\rho\defeq \Tst^*({\cM})\otimes_{\cO_E}E$ and $\Dst^{*}(\rho)$ be the associated contravariant filtered $(\varphi,N)$-module, we have an isomorphism
\begin{equation}
\label{relation descent}
{\cM}\otimes_{S_{\cO_E}}E\cong \Dst^{*}(\rho)
\end{equation}
via the base change $S_{\cO_E}\rightarrow E$ defined by $u\mapsto 0$.
\end{theo}

As for Breuil modules, we can define the category of \emph{quasi-strongly divisible} $R$-modules $\RModddN[r]$, where we omit the requirement for a monodromy operator (cf. \cite{liu-breuil}, \S 2). Again we have a contravariant functor $\Tqst^{*}$ towards the category of $G_{(\Qp)_{\infty}}$-representations over $\cO_E$, inducing an anti-equivalence:
\begin{theo}[\cite{liu-breuil}]
\label{liu-quasi}
The contravariant functor
$$
\Tqst^*:\OEModddN[r]\rightarrow \mathrm{Rep}^{\text{$K$-st}, [-r,0]}_{\Oe}(G_{(\Qp)_{\infty}})
$$
establishes an anti-equivalence of categories if $r<p-1$.
\end{theo}
Here, we wrote $\mathrm{Rep}^{\text{$K$-st}, [-r,0]}_{\Oe}(G_{(\Qp)_{\infty}})$ to denote the category of $G_{(\Qp)_{\infty}}$-stable $\cO_E$-lattices inside $E$-valued, finite dimensional $p$-adic Galois representation of $G_{\Qp}$ becoming {semi-stable} over $K$ and with Hodge--Tate weights in $[-r,0]$.

We will be mainly concerned with the covariant version of the above functors towards Galois representations. 
For this reason we define $\Tst^{r}:\RModdd[r]\rightarrow\mathrm{Rep}_{R}(G_{\Qp})$ and $\Tst^{r}:\RBrModdd[r]\rightarrow\mathrm{Rep}_{\Rbar}(G_{\Qp})$ via \begin{eqnarray*}
\Tst^{r}(\widehat{\cM})\defeq \left(\Tst^{*}(\widehat{\cM})\right)^{\vee}\otimes\varepsilon_p^r,&\qquad&
\Tst^{r}({\cM})\defeq \left(\Tst^*(\cM)\right)^{\vee}\otimes\omega^r
\end{eqnarray*}
respectively (where we write $\bullet^{\vee}$ to denote the usual linear dual).

\section{Ordinary strongly divisible lattices}
The aim of this section is to describe quasi-strongly divisible lattices (with a carefully chosen descent datum) associated to potentially crystalline lift of an ordinary residual Galois representation. We first study the filtration modulo $p$ and then lift to $\cO_E$.  We perform a $p$-adic convergent argument generalizing the technique from \cite[\S 5]{breuil-buzzati} for $\GL_2$ to diagonalize the Frobenius in this setting. The main result is Theorem \ref{phifilr}.

Let $\rhobar : G_{\qp}\rightarrow\gl{3}{\F}$ be a continuous Galois representation. We assume that $\rhobar$ is \emph{ordinary}, of the form
\begin{equation*}
\label{general form Galois representation}
\galorep\sim\maqn{ccc}{\omega^{a_2+2}\mu_{\alpha_2}&\ast&\ast\\
0&\omega^{a_{1}+1}\mu_{\alpha_{1}}&\ast\\
0&0&\omega^{a_0}\mu_{\alpha_0}
}
\end{equation*}
where $\mu_{\alpha_i}$ denotes the unramified character on $\Zp$ verifying $\mu_{\alpha_i}(p)=\alpha_i\in E^{\times}$ and where the exponents $a_i\in\N$ verify
\begin{equation}
\label{strong genericity}
a_1-a_0,\,a_2-a_1>3,\quad\text{and}\,a_2-a_0<p-4
\end{equation}
Provided the conditions (\ref{strong genericity}) we say that $\rhobar$ is \emph{strongly generic}.

It will be convenient to introduce the following notation for the subquotients of $\rhobar$: for $i\in\{0,1\}$ (resp. $j\in\{0,1,2\}$) we define: 
\begin{eqnarray*}
\rhobar_{j}\defeq  \omega^{a_j+j}\mu_{\alpha_{j}},&\qquad&
\rhobar_{i,i+1}\defeq \maq{\omega^{a_{i+1}+i+1}}{\ast}{0}{\omega^{a_i+1}}.
\end{eqnarray*}
With this formalism, we have $\rhobar_{0,2}\defeq\rhobar$.

\subsection{Filtration on ordinary quasi-strongly divisible modules}
\label{sectionfiltration}

We show here that the relative position of the descent data $\tau$ and the inertial weights of $\rhobar$ provide strong constraints on the filtration of strongly divisible lattices lifting $\rhobar$. The main result is Theorem \ref{mainFiltr}, which is the first step in the proof of Theorem \ref{phifilr}.

We keep the notations of the previous section. In particular,  $\rhobar:G_{\Qp}\rightarrow\rGL_3(\F)$ is as in section \ref{notations}. In all what follows, $\cM$ (resp. $\barcM$) denotes a quasi-strongly divisible module (resp. quasi-Breuil module) with descent data such that $\Tqst^{2}(\cM)$ is a lift of $\rhobar\vert_{G_{(\Qp)_{\infty}}}$ (resp. such that $\Tst^2(\barcM)\cong \rhobar\vert_{G_{(\Qp)_{\infty}}}$). We fix once and for all a niveau 1 descent data $\tau$ on $\cM$ (resp. $\barcM$), of the form $\tau\defeq \teich{\omega}^{a_0}\oplus\teich{\omega}^{a_1}\oplus\teich{\omega}^{a_2}$. We refer to $\tau$ as a \emph{principal series type}. 

If $\mathcal{M}$ is such a module, we have a $S_{\mathcal{O}_E}$-basis $\underline{e}=(e_0,e_1,e_2)$ which is compatible with the action of $\Delta$ (the tame descent). The goal of this section is to prove the following result describing the filtration $\mathrm{Fil}^2\mathcal{M}$:

\begin{theo}
\label{mainFiltr}
Let $\Tqst^{2}(\mathcal{M})\defeq \rho$ be a lift of $\rhobar\vert_{G_{(\Qp)_{\infty}}}$, with principal series type $\tau$. Assume that the integers $a_0,a_1,a_2$ verify the strong genericity assumption (\ref{strong genericity}).

There exists an $S_{\mathcal{O}_E}$-basis $(e_0,e_1,e_2)$ for $\mathcal{M}$, compatible the residual Galois action, such that
$$
\mathrm{Fil}^2\mathcal{M}=\langle e_0,\,E(u)e_1,\,E(u)^2e_2\rangle_{\SBr}+\mathrm{Fil}^p\SBr\cdot \mathcal{M}.
$$
\end{theo}
The remainder of this section is devoted to the proof of Theorem \ref{mainFiltr} and the first step is to study the filtration on the associated quasi-Breuil module $\overline{\mathcal{M}} := \mathcal{M} \otimes_{S_{\cO_E}} \overline{S}$. We start from recalling results from \cite{HM} concerning quasi-Breuil modules and their subobjects. Recall that $K_0=\Qp$ in our setting.

\begin{def}
\label{definition:qBreuil submodule}
Let $\barcM$ be an object in $\FBrModddN[r]$. An $\barS$-submodule $\barcN\subseteq\barcM$ is said to be a \emph{quasi-Breuil submodule} if $\barcN$ fulfills the following conditions:
\begin{itemize}
	\item[$i)$] $\barcN$ is an $\barS$-direct summand in $\cM$;
	\item[$ii)$] $\barcN$ is stable under the descent data;
	\item[$iii)$] the Frobenius $\varphi_r$ on $\Fil^r\barcM$ restricts to a $\varphi$-semilinear morphism $\barcN\cap\Fil^r\barcM\rightarrow \barcN$.
\end{itemize}
\end{def}

The relevant properties concerning quasi-Breuil submodules are summarized in the following statements. Their proofs are all contained in \cite{HM}, Appendix A.

\begin{prop}
Let $\cM$ be an object in $\FBrModddN$ and let $\barcN\subseteq\barcM$  be a quasi-Breuil submodule.
Then the $\barS$-modules $\barcN$,  $\barcM/\barcN$ are naturally objects in $\FBrModddN[r]$ and the sequence
$$
0\rightarrow \barcN\rightarrow \barcM\rightarrow \barcM/\barcN\rightarrow 0
$$
is exact in $\FBrModddN[r]$.

Moreover, with the above notion of exact sequence, the category $\FBrModddN[r]$ is an exact category and  $\Tqst^{r}$ is an exact functor. 
\end{prop}

From \cite{HM}, Proposition 2.2.4 and 2.2.5, we deduce the following important result: 
\begin{prop}
\label{bijection ordrer preserving}
Let $\barcM\in\FBrModddN[r]$ be a quasi-Breuil module. The functor $\Tqst^{r}$ induces an order preserving bijection:
\begin{eqnarray*}
\Theta:\left\{\text{quasi-Breuil\,submod.\,in\,}\barcM\right\}\stackrel{\sim}{\rightarrow}
\left\{G_{(\Qp)_{\infty}}\text{\,-subrep.\,in\,}\Tqst^{r}(\barcM)\right\}
\end{eqnarray*}
which canonically identifies $\Theta(\barcM)/\Theta(\barcN)$ with $\Tqst^{r}(\barcM/\barcN)$ for any quasi-Breuil submodule $\barcN\subseteq \barcM$.
\end{prop}

By Proposition \ref{bijection ordrer preserving}, for $0\leq i\leq j\leq 2$  there are unique subquotients $\barcM_{i,j}$ of $\barcM$ such that $\Tqst^{2}(\barcM_{i,j})\cong \rhobar_{i,j}\vert_{G_{(\Qp)_{\infty}}}$.

In particular, we have $\Tqst^{2}(\barcM_{i,i}) \cong \omega^{a_i + i} \mu_{\alpha_i}\vert_{G_{(\Qp)_{\infty}}}$ for all $0\leq i\leq 2$.

\begin{lem} 
\label{oned}
Assume that $\rhobar$ is strongly generic $($\ref{strong genericity}$)$.
For $0 \leq i \leq 2$ we have 
$$
\Fil^2 \barcM_{i,i} = u^{ie} \barcM_{i,i}
$$
and $\barcM_{i,i} = (\F[u]/u^{ep}) {e}_i$ where $\Delta$ acts on ${e}_i$ by $\omega^{a_i}$.  
\end{lem}
\begin{proof}
Specializing \cite[Lemma 3.3.2]{EGS} to our situation, we have that $\Fil^2 \barcM_{i,i} = u^{r} \barcM_{i,i}$, and $\barcM_{i,i} = (\F[u]/u^{ep}) e_i$ where $\Delta$ acts on $e_i$ by the character $\omega^{k}$ with $0 \leq r \leq 2(p-1)$ and $k \equiv p(k + r) \mod p-1$.  Furthermore, since $\rhobar_{i,i}|_{I_{\qp}} = \omega^{a_i + i}$, we have
$$
a_ i  + i \equiv k + p \frac{r}{p-1} \mod p - 1. 
$$
By the first congruence, $p - 1 \mid r$ so $r  \in \{0, p-1, 2(p-1) \}$.  Since $\barcM$ has tame descent given by $\tau$, we know $k \in \{a_j\}_{j = 0, 1, 2}$.  The second congruence becomes
$$
a_i - a_j + i \equiv 0, 1 \text{ or }2 \mod p-1
$$
which is not possible unless $i =j$ by our genericity assumption.      
\end{proof} 

\begin{prop}
\label{corofiltrazione}
Assume $\rhobar\vert_{G_{(\Qp)_{\infty}}}\defeq \Tqst^{2}(\barcM)$ is strongly generic. There exists a basis $(e_0,e_1,e_2)$ on $\barcM$, compatible with the descent data, such that $\Fil^2\barcM$ is described by
\begin{equation*}
\Fil^2\barcM=\langle e_0,\,u^ee_1,\,u^{2e}e_2\rangle_{\barS}.
\end{equation*}

\end{prop}
\begin{proof}
We fix a basis $\underline{e}=(e_0,e_1,e_2)$ on $\barcM$, compatible with the descent data. We use the same notation for the image of the elements $e_i$ in the various subquotients of $\barcM$, this will cause no confusion.  By Lemma \ref{oned}, $\barcM_i$ is generated by $e_i$ and $\Fil^2 \barcM_i = u^{ie} \barcM_i$ for all $i$. 

From the exact sequences $0\rightarrow \Fil^2\barcM_{i+1}\rightarrow\Fil^2\barcM_{i,i+1}\rightarrow\Fil^2\barcM_i\rightarrow0$ and Lemma \ref{oned} it follows that $\Fil^{2}\barcM$  admits a system of generators $\underline{f}=(f_0,f_1,f_2)$, compatible with the descent data, such that
\begin{equation*}
\mathrm{Mat}_{\underline{e},\underline{f}}\left(\Fil^2\barcM\right)=
\maqn{ccc}{1&0&0\\(x_0 + x_1 u^e) u^{e-(a_1-a_0)}&u^e&0\\ (z_0+z_1u^e)u^{e-(a_2-a_0)}& (y_0+y_1u^e)u^{e-(a_2-a_1)}&u^{2e}}
\end{equation*}
where $x_i, y_i,z_i\in \F$. We can assume that $x_1=0$ and write $x=x_0$ in what follows. Indeed, as $e-(a_1-a_0)>0$, Nakayama's lemma shows that $(f_0-x_1u^{e-(a_1-a_0)}f_1, f_1, f_2)$ is still a system of generators for $\Fil^2\barcM$, compatible with the descent data.  We also note that $y_0 = 0$ since $\Fil^2 \barcM \supset u^{2e} \barcM$ (in particular, $u^{2e}e_1\in \Fil^2\barcM$).


We finally deduce that the $\barS$-module $\Fil^2\barcM$ is generated by the following elements (described in their coordinates with respect to the basis $(e_0, e_1,e_2)$):
\begin{equation*}
f_0\defeq \vectln{lll}{1\\xu^{e-(a_1-a_0)}\\ (z_0+z_1u^e) u^{e-(a_2-a_0)}},\qquad
f_1\defeq u^e\vectln{lll}{0\\1\\yu^{e-(a_2-a_1)}},\qquad f_2\defeq u^{2e}\vectln{lll}{0\\0\\1}.
\end{equation*}
By letting $e_0'\defeq f_0$, $e_1'\defeq e_1+yu^{e-(a_2-a_1)} e_2$, we immediately see that $(e_0',e_1',e_2)$ is a basis on $\barcM$ with the required properties.
\end{proof}

Thanks to Proposition \ref{corofiltrazione}, we are able to describe the filtration $\mathrm{Fil}^r$ on the quasi-strongly divisible  module $\mathcal{M}$.
\vspace{2mm}

\paragraph{\textbf{Proof of Theorem \ref{mainFiltr}.}}

Let $(\tilde{m}_0,\tilde{m}_1,\tilde{m}_2)$ be an $S\defeq \SBr$ basis for $\mathcal{M}$, compatible with the descent data. Let $(e_0, e_1, e_2)$ denote a basis for $\barcM$ satisfying the condition from Proposition \ref{corofiltrazione}.

Since 
$$
\left(\mathrm{Fil}^2\mathcal{M}/\mathrm{Fil}^2S\cdot \mathcal{M}\right)_{\teich{\omega}^{a_0}}
\twoheadrightarrow \left(\mathrm{Fil}^2\barcM/\Fil^2 S\cdot \barcM\right)_{\teich{\omega}^{a_0}}=
\langle e_0, u^{e-(a_1-a_0)}u^ee_1\rangle_{\barS_{\omega^0}}+u^{2e}\cdot \left(\mathcal{M}\right)_{\teich{\omega}^{a_0}}
$$
we have a lift $\tilde{e}_0\in\mathrm{Fil}^2\mathcal{M}$ of $e_0$. Notice that, by Nakayama's lemma, the family $(\tilde{e}_0,\tilde{m}_1,\tilde{m}_2)$ is again a basis for $\mathcal{M}$, compatible with the descent data.

Let $\mathcal{D}\defeq \cM\otimes_{\Zp}\Qp$.
By virtue of \cite{HM} Lemma 2.3.9, the $\teich{\omega}^{a_0}$-isotypical component  $\left(\mathrm{Fil}^2\mathcal{D}/\mathrm{Fil}^2S_E\cdot \mathcal{D}\right)_{\teich{\omega}^{a_0}}$ is described as follows in terms of coordinates with respect to the basis $(\tilde{e}_0,\tilde{m}_1,\tilde{m}_2)$:
$$
\big(\mathrm{Fil}^2\mathcal{D}/\mathrm{Fil}^2S_E\cdot\mathcal{D}\big)_{\teich{\omega}^{a_0}}=\left\langle
\vectln{lll}{1\\ 0\\ 0}, E(u)\vectln{lll}{1\\0\\0}, 
E(u)\vectln{lll}{0\\b_1'u^{e-(a_1-a_0)}\\b_2'u^{e-(a_2-a_0)}}\right\rangle_{E}
$$
for some $(b_1',b_2')\in\mathbf{P}^1_{E}(E)$. 

The $\mathcal{O}_E$-saturation of the latter space is now easy to determine, and we get
$$
\big(\mathrm{Fil}^2\mathcal{M}/\mathrm{Fil}^2\SBr\cdot\mathcal{M}\big)_{\teich{\omega}^{a_0}}=\left\langle
\vectln{lll}{1\\ 0\\ 0}, E(u)\vectln{lll}{1\\0\\0}, 
E(u)\vectln{lll}{0\\b_1 u^{e-(a_1-a_0)}\\b_2 u^{e-(a_2-a_0)}}\right\rangle_{\cO_E}
$$
for some $(b_1,b_2)\in\mathbf{P}^1_{\mathcal{O}_E}(\mathcal{O}_E)$.

Moreover, the $\teich{\omega}^{a_1}$-isotypical component of $\mathrm{Fil}^2\mathcal{D}/\mathrm{Fil}^2S_E\cdot\mathcal{D}$ is described by
$$
\left(\mathrm{Fil}^2\mathcal{D}/\mathrm{Fil}^2S_E\cdot\mathcal{D}\right)_{\teich{\omega}^{a_1}}=\left\langle
\vectln{lll}{u^{a_1-a_0}\\ 0\\ 0}, E(u)\vectln{lll}{u^{a_1-a_0}\\0\\0}, 
E(u)\vectln{lll}{0\\c_1'\\c_2'u^{e-(a_2-a_1)}}\right\rangle_{E}
$$
for some $(c_1',c_2')\in\mathbf{P}^1_{E}(E)$, hence
$$
\left(\mathrm{Fil}^2\mathcal{M}/\mathrm{Fil}^2\SBr\cdot\mathcal{M}\right)_{\teich{\omega}^{a_1}}=\left\langle
\vectln{lll}{u^{a_1-a_0}\\ 0\\ 0}, E(u)\vectln{lll}{u^{a_1-a_0}\\0\\0}, 
E(u)\vectln{lll}{0\\c_1 \\c_2 u^{e-(a_2-a_1)}}\right\rangle_{\mathcal{O}_E}
$$
for some $(c_1,c_2)\in\mathbf{P}^1_{\mathcal{O}_E}(\mathcal{O}_E)$.

In particular, we have 
$$
\left(\mathrm{Fil}^2\mathcal{M}/\Fil^2S\cdot \cM\right)_{\teich{\omega}^{a_1}}\subseteq\left\langle \vectln{lll}{u^{a_1-a_0}\\ 0\\ 0}\right\rangle_{\mathcal{O}_E}+E(u)\cdot \big(\mathcal{M}\big)_{\teich{\omega}^{a_1}},
$$
and since 
$$
\left(\mathrm{Fil}^2\mathcal{M}/\mathrm{Fil}^2S\cdot \mathcal{M}\right)_{\teich{\omega}^{a_1}}
\twoheadrightarrow \big(\mathrm{Fil}^2\overline{\mathcal{M}}/\Fil^2\barS\cdot\barcM\big)_{\teich{\omega}^{a_1}}=
\langle u^{a_1-a_0}e_0, u^ee_1\rangle_{\barS_{\teich{\omega}^0}}+u^{2e}\left(\mathcal{M}\right)_{\teich{\omega}^{a_1}},
$$
a lift $f_1\in\big(\mathrm{Fil}^2\mathcal{M}\big)_{\teich{\omega}^{a_1}}$ 
of a generator of $\big(\mathrm{Fil}^2\overline{\mathcal{M}}/\Fil^2\barS\cdot\barcM\big)_{\teich{\omega}^{a_1}}$
has the form $f_1= \lambda_0u^{a_1-a_0}\tilde{e}_0+E(u)\tilde{e}_1$ for some 
$\lambda_0\in\mathcal{O}_E$ and some element $\tilde{e}_1\in \big(\mathcal{M}\big)_{\teich{\omega}^{a_1}}$.

We deduce that $E(u)\tilde{e}_1\in \big(\mathrm{Fil}^2\mathcal{M}\big)_{\teich{\omega}^{a_1}}$, as well as $E(u)\tilde{e}_1\equiv E(u)e_1$ modulo $(\varpi_E,\mathrm{Fil}^2\SBr)$. Hence, $\widetilde{e}_1\equiv e_1$ modulo $(\varpi_E, u^{e(p-1)})$.
It follows that $(\tilde{e}_0,\tilde{e}_1,\tilde{m}_2)$ is again a basis for $\mathcal{M}$ (compatible with the descent data). In terms of coordinates with respect the basis $(\tilde{e}_0,\tilde{e}_1,\tilde{m}_2)$, we now have
$$
\big(\mathrm{Fil}^2\mathcal{M}/\mathrm{Fil}^2\SBr\cdot\mathcal{M}\big)_{\teich{\omega}^{a_0}}=\left\langle
\vectln{lll}{u^{[a_0-a_1]}\\ 0\\ 0}, E(u)\vectln{lll}{u^{[a_0-a_1]}\\0\\0}, 
E(u)\vectln{lll}{0\\1 \\0}\right\rangle_{\mathcal{O}_E}.
$$

If we now let $\tilde{e}_2\defeq \tilde{m}_2$ it is elementary to conclude that, in the basis $(\tilde{e}_0,\tilde{e}_1,\tilde{e}_2)$, we have
$$
\mathrm{Fil}^2\mathcal{M}=\langle E(u)^j \widetilde{e}_j,\quad j\in\{0,1,2\}\rangle_{\SBr}+\mathrm{Fil}^p\SBr\cdot \mathcal{M}.
$$
as claimed.

\subsection{Diagonalization of the Frobenius action}
\label{sectFrob}

The aim of this section is to provide quasi-strongly divisible modules $\cM$ lifting $\overline{\cM}$ with a gauge basis (i.e., a basis for the filtration on which the Frobenius is diagonal).  We keep the notations from previous sections. In particular, we let $\mathcal{M}$ denote a quasi-strongly divisible modules over $S_{\mathcal{O}_E}$ with principal series descent data $\tau$ and such that $\Tqst^{2}\left(\cM\right)\otimes_{\cO_E}\F\cong \rhobar\vert_{G_{(\Qp)_{\infty}}}$. Let us fix a framing on $\tau$, i.e. a basis of eigenvector on the underlying vector space of $\tau$.

From \cite{HM}, Lemma 2.2.7 (and the fact that $\SBr$ is a local ring) we deduce the existence of a basis $\underline{e}\defeq (e_0,e_1,e_2)$ for $\cM$ and a system of generators $\underline{f}\defeq (f_0,f_1,f_2)$ for $\Fil^2\cM$ modulo $\Fil^p \SBr \cM$
such that
\begin{equation*}
\widehat{g}(e_i)=\teich{\omega}^{a_i}(g)e_i,\qquad 
\widehat{g}(f_i)=\teich{\omega}^{a_i}(g)f_i
\end{equation*}
for all $i\in\{0,1,2\}$. We say that $\underline{e},\,\underline{f}$ are \emph{compatible} with the descent datum $\tau$.

The main result of this section is a complete description of $\Fil^2\cM/(\Fil^p \SBr\cdot \cM)$ and the Frobenius action $\varphi_r$ on $\cM$ in terms of a framed basis $\underline{e}$, and generators $\underline{f}$.

\begin{theo}
\label{phifilr}
Let $\cM$ be a quasi-strongly divisible lattice with tame descent data of type $\tau\cong \teich{\omega}^{a_0}\oplus\teich{\omega}^{a_1}\oplus\teich{\omega}^{a_2}$ where $a_0, a_1, a_2$ are strongly generic $($\ref{strong genericity}$)$. Fix a framing on $\tau$, a framed basis $\un{e}=(e_0,e_1,e_2)$ on $\cM$ and assume that $\Fil^2\cM/(\Fil^p \SBr\cdot \cM)$ is generated by $(e_0,E(u)e_1,E(u)^2e_2)$.

Then there exists a basis $\underline{e}^{(\infty)}\defeq (e_0^{(\infty)},e_1^{(\infty)},e_{2}^{(\infty)})$ for $\mathcal{M}$, and a system of generators 
$\underline{f}^{(\infty)}\defeq (f_0^{(\infty)},f_1^{(\infty)},f_{2}^{(\infty)})$ for $\Fil^2\cM/(\Fil^p \SBr\cdot \cM)$ compatible with the framing on $\tau$ and such that:
\begin{eqnarray*}
\Mat_{\un{e}^{(\infty)}}([f_0^{(\infty)},\,f_1^{(\infty)},\,f_{2}^{(\infty)}])&=&\small\begin{pmatrix}
1&0&0
\\ 
u^{e-(a_1-a_0)}v_{1,0}^{(\infty)}&E(u)&0
\\ 
u^{e-(a_2-a_0)} v^{(\infty)}_{2,0}&u^{e-(a_2-a_1)} E(u)v^{(\infty)}_{2,1}&E(u)^2
\end{pmatrix}\\
\Mat_{\un{e}^{(\infty)},\un{f}^{(\infty)}}(\phz_2)&=&\begin{pmatrix}\lambda_0&0&0\\0&\lambda_1&0\\0&0&\lambda_2\end{pmatrix}
\end{eqnarray*}
where $\lambda_i\in \cO_E^{\times}$ and $v_{1,0}^{(\infty)},\, v_{2,1}^{(\infty)}\in \cO_E$, $v_{2,0}^{(\infty)}\in \cO_E\oplus E(u)\cO_E$.
\end{theo}

The proof of theorem \ref{phifilr} relies on a delicate $p$-adic convergence argument and occupies the remainder of this section. 

\begin{rmk}
\begin{enumerate}
\item The statement and proof of Theorem \ref{phifilr} generalizes, \emph{mutatis mutandis}, to an $n$ dimensional ordinary representation. Because of the technicality of the computations in the $\mathbf{GL}_n$ case, and for sake of readability, we focus only on the case $n=3$ $($so $r = 2)$.

\item Similarly, Theorem \ref{phifilr} is stated for strongly divisible lattices with $\cO_E$-coefficient, but the statement and the proof generalizes line to line when $\SBr$ is replaced by $S_R$, where $R$ is a local, complete noetherian $\cO_E$-algebra which is $p$-flat (cf. with \cite{breuil-buzzati} Proposition 5.4 and \cite{EGS}, beginning \S7.4, where the similar situation in the $\rGL_2$ case is discussed).
\end{enumerate}
\end{rmk}

\subsubsection{Linear algebra with coefficients}
\label{LAC new}

The goal of this section is to develop and collect results of linear algebra with coefficients over a certain family of closed subrings of $\SBr$. It is the natural generalization of \cite{HM}, \S 2.2.3 in characteristic zero.

In all what follows, we take $r = 2$ though the results can be generalized to $r < p-1$. Define
$$
\mathscr{R}\defeq \left(\SBr\right)_{\teich{\omega}^0}=\widehat{\bigoplus}_{i\geq 0}\mathcal{O}_E\cdot\frac{u^{ie}}{i!}.
$$
We have a natural filtration $\Fil^i\sR\defeq \Fil^i\SBr\cap \sR$ on $\sR$ and we note that $\sR$ is stable under the Frobenius on $\SBr$.
As $\sR$ is complete and separated for the $p$-adic topology we have
$$
\bigcap_{n\in\N}\left(\cO_E+p^{n+1}\sR\right)=\mathcal{O}_E.
$$


The following closed ideals will be important for the $p$-adic convergence argument.
\begin{eqnarray*}
\cJ \defeq \left( p\Fil^2\sR, \Fil^{3}\sR\right)=\Fil^2\sR\cdot (p,\Fil^1\sR),&\qquad& \cI\defeq p\Fil^1\sR.
\end{eqnarray*}
We collect some important lemmas on the nature of the filtration on $\sR$.
\begin{lem}
\label{lemcalcul1 new} 
Let $p>n\geq 1$ and $i\geq n$ be integers.
We have the following relation in the Breuil ring $\sR$:
$$
\frac{u^{ie}}{i!}\in \left(\bigoplus_{k=0}^{n-1}p^{n-k}\mathcal{O}_E E(u)^{k}\right)+ E(u)^n\cdot \left(\widehat{\bigoplus}_{i\geq 0}\mathcal{O}_E\cdot E(u)^i\right).
$$

In particular
for $3\leq i\leq p-1$ we have
$$
\phz_2(u^i)\in u^i\cdot\left(p^i+p\cI+\cJ\right)
$$
and, more generally 
$$
\sR=\left(\cO_E\oplus\cO_EE(u)\oplus\cO_E E(u)^2\right)+\cI+\cJ.
$$
\end{lem} 
\begin{proof}
In all what follows we work modulo $E(u)^n\cdot \sR$. We have:
$$
\frac{u^{ei}}{i!}=\sum_{k=0}^{n-1}\frac{E(u)^k}{k!}\frac{1}{(i-k)!}(-p)^{i-k}
$$
and hence we are done once we can show that
$$
(i-k)-\frac{(i-k)-\mathrm{S}_p(i-k)}{p-1}\geq n-k
$$
where $\mathrm{S}_p(t)$ denotes the sum of the digits in the $p$-adic development of $t\in\N$.

If $i\leq p-1$ this is obvious.
Else, we show that $(i-k)-\frac{(i-k)-\mathrm{S}_p(i-k)}{p-1}\geq p-k$.  Defining $a_s\in\{0,\dots,p-1\}$ via $i-k=\sum_{s\geq 0}a_sp^s$ we have
$$
(i-k)-\frac{(i-k)-\mathrm{S}_p(i-k)}{p-1}=a_0+\sum_{s\geq 1}a_s\bigg(p^s-\frac{p^s-1}{p-1}\bigg);
$$
if $i=\sum_sb_sp^s\geq p>k$ with $b_s\in\{0,\dots,p-1\}$ then either $a_0=b_0-k\geq 0$ and $b_s\neq 0$ for some $s\geq1$ or $a_0=p+b_0-k$.

The last two statements now follow: from the above and the definition of $\cJ$, $\cI$ we have 
$$
\widehat{\bigoplus}_{i\geq 3}\mathcal{O}_E\cdot\frac{u^{ie}}{i!}\subseteq \cJ+p\cI+p^3\cO_E
$$
and $u^{2e}\in \cO_EE(u)^2+\cI+p^2\cO_E$, $u^e\in \cO_EE(u)+p\cO_E$.
\end{proof}

We introduce below the formalism of linear algebra with coefficients. It is the characteristic zero version of the formalism introduced in \cite{HM}, \S 2.3.2.

\begin{defi} 
Let $0 \leq a_0\leq a_1\leq a_2 \leq e$ be the integers associated to the niveau one descent data. 
For a pair $(a_i,a_j)$ let $\left[ a_i-a_j\right]\in\{0,\dots,e-1\}$ be defined by $\left[ a_i-a_j\right]\equiv -(a_i-a_j)$ modulo $e$. The $\mathscr{R}$-module of \emph{matrices with descent data} is defined as:
\begin{eqnarray*}
\mathrm{M}_{dd,3}(\mathscr{R})\defeq\left\{M\in \mathrm{M}_{3}(\SBr),\,\mathrm{s.t.}\,\,
M_{i,j}=u^{\left[ a_i-a_j\right]}m_{i,j}\,\,\mathrm{with}\,\,m_{i,j}\in\mathscr{R}
\right\}
\end{eqnarray*}
\end{defi}
The following result is an elementary check in linear algebra:
\begin{lem} The subset $\mathrm{M}_{dd,3}(\mathscr{R})$ is a subring of $\mathrm{M}_{3}(\SBr)$. Moreover, if $M\in \mathrm{M}_{dd,3}(\mathscr{R})$, then the adjugate matrix  $M^{\mathrm{adj}}$ is again an element of  $\mathrm{M}_{dd,3}(\mathscr{R})$.
\end{lem}

We introduce certain natural subsets of $\rM_{dd,3}(\mathscr{R})$:

 \begin{itemize}
 \item $\rGL_{dd,3}(\mathscr{R})$ the group of invertible elements in $\rM_{dd,3}(\mathscr{R})$; 
 \item $\rB_{dd,3}(\sR) \subset \rGL_{dd,3}(\mathscr{R})$ the subgroup of upper triangular matrices;
 \item $\rU^{\opp}_{dd, 3}(\sR) \subset \rGL_{dd,3}(\mathscr{R})$  the subgroup of strictly lower triangular unipotent matrices;
 \item $\rL_{dd,3}(\sR) \subset \rM_{dd,3}(\mathscr{R})$ the multiplicative monoid of lower triangular matrices.
  \item $\rT_{3}(\cO_E)$ is the subgroup of diagonal matrices with scalar entries.
 \end{itemize} 

If $M_1,\, M_2\in \rM_{dd,3}(\sR)$ and $\mathcal{K}$ is an ideal of $\sR$ it is customary to write:
\begin{equation*}
M_1\equiv M_2\,\,\mathrm{mod}\,\mathcal{K}
\end{equation*}
to mean that $M_1=M_2+M'$ for some $M'\in \rM_{dd,3}(\mathcal{K})$. Also, if $M\in \rM_{dd,3}(\mathscr{R})$ we write $M^{\mathrm{adj}}$ to denote its adjugate.

We record two elementary manipulations between matrices with descent data:
\begin{lem}
\label{lemconto3.5 new}
Let $A=\left(u^{\left[ a_i-a_j\right]}a_{i,j}\right)_{i,j}\in \rM_{dd,3}(\sR)$ and $W=\left(u^{\left[ a_i-a_j\right]}w_{i,j}\right)_{i,j}\in \rM_{dd,3}(\mathscr{R})$. Assume moreover that $W \in \rU^{\opp}_{dd, 3}(\sR) $. Then
$$
\left(W\cdot A\right)_{i,j}=\left\{\begin{array}{cc}
u^{\left[ a_i-a_j\right]}\bigg(a_{i,j}+u^e\big(\sum_{k=0}^{i-1}w_{i,k}a_{k,j}\big)\bigg)&\text{if}\,\,j\geq i\\
u^{\left[ a_i-a_j\right]}\bigg(a_{i,j}+u^e\big(\sum_{k=j+1}^{i-1}w_{i,k}a_{k,j}\big)+\sum_{k=0}^{j}w_{i,k}a_{k,j}\bigg)&\text{if}\,\,j< i\\
\end{array}\right.
$$
\end{lem}
\begin{proof}
Omitted.
\end{proof}

\begin{lem}
\label{lem3.6 new}
Let $W=\left(u^{\left[ a_i-a_j\right]}w_{i,j}\right)_{i,j}$ be an element in $\rU^{\opp}_{dd, 3}(\sR)$. Assume that $w_{i,j}\in \bigoplus_{k=0}^{i-j}E(u)^{k}\mathcal{O}_E$ for all $0\leq j\leq i\leq 2$.  Then $$W^{\mathrm{adj}}=\left(u^{\left[ a_i-a_j\right]}w^{\mathrm{adj}}_{i,j}\right)_{i,j} \in \rU^{\opp}_{dd, 3}(\sR) $$ satisfies
$w^{\mathrm{adj}}_{i,j}\in \bigoplus_{k=0}^{i-j}E(u)^{k}\mathcal{O}_E$ 
for all $0\leq j\leq i\leq 2$.
\end{lem}
\begin{proof}
Omitted.
\end{proof}

The following Lemma plays a crucial role in the $p$-adic convergence argument. It describes the effect of the Frobenius on the elements in $\rM_{dd,3}(\sR)$.

\begin{lem}
\label{lemphir new new}
Let $n\geq 0$ and let $M\in\rGL_{dd,3}(\mathscr{R})$. 
If $n\geq 1$, assume further that $M\in \rT_3(\cO_E)+\rM_{dd,3}(p^n\sR)$. 
\begin{itemize}
	\item[$i)$] If $[a_i-a_j]\geq 3$ for all $0\leq i,j\leq 2$ then 
	$$
	\varphi\left(M\right)\in \rT_3(\cO_E)+\rM_{dd,3}(p^n\cJ+p^n\cI+p^{n+1}\sR)
	$$
	and $M \equiv \varphi\left(M\right) \mod p^n\sR$.
	\item[$ii)$] If $M \in\mathbf{T}_{3}(\mathcal{O}_E)+\rM_{dd,3}\left(p^n(p,\mathrm{Fil}^1\sR)\right)$ then 
	$$
	\varphi\left(M\right)\in \mathbf{T}_{3}(\mathcal{O}_E)+\rM_{dd,3}\left(p^{n+1}\sR\right)
	$$
	and  $M \equiv \varphi\left(M\right) \mod p^n\sR$. 
\end{itemize}
\end{lem}
\begin{proof}
As $\phz(E(u))\in p\sR^{\times}$, claim $ii)$ is obvious for any $n\geq 0$.

Let us consider $i)$. 
Note first that if $x\in \sR^{\times}$, then we can always write $x\in x_0+(p,\ \Fil^1\sR)$ for some $x_0\in \cO_E^{\times}$.
If $A\in \rM_{dd,3}(\sR)$ and since we assume $[a_i-a_j]\geq 3$ for all $0\leq i,j,\leq 3$, we deduce from Lemma \ref{lemcalcul1 new} that $\varphi\left(A\right)\in \rM_{dd,3}(\cJ+p\cI+p^3\sR)$ as soon as $A_{(i,i)}\in \Fil^1\sR$.

Hence, if $M\in\rGL_{dd,3}(\mathscr{R})$ if $n=0$ (resp. $M\in\rT_3(\cO_E)+\rM_{dd,3}(p^n\sR)$ when $n\geq 1$) we can always write $M\in T_M+A+\rM_{dd,3}(p^n(p,\Fil^1\sR))$ for some $A\in \rM_{dd,3}(p^n\sR)$ verifying $A_{(i,i)}\in p^n(p,\Fil^1\sR)$ and some $T_M\in \rT(\cO_E)$; the first claim follows. 

\end{proof}

\subsection{Proof of Theorem \ref{phifilr}}

The proof of Theorem \ref{phifilr} is a $p$-adic convergence procedure. It involves an induction argument which consists in a careful change of basis on $\cM$.  We again specialize to the case of $n = 3$ (so $r = 2$) though the procedure works more generally for ordinary families. We continue to assume that the triple $(a_0,a_1,a_2)$ is strongly generic.

If $\underline{e}=(e_0,e_1,e_2)$, $\underline{f}=(f_0,f_1,f_2)$ are a basis for $\cM$ and a generating family of $\Fil^2\cM/(\Fil^p \SBr \cdot \cM)$, which are compatible with the framing on $\tau$, we define the element 
 $V=V_{\underline{e},\underline{f}}\in \rM_{dd,3}(\sR)$ such that
$$
f_0 = \underline{e}\cdot V \vectln{c}{1\\0\\0},\,\, f_1 = \underline{e}\cdot V \vectln{c}{0\\1\\0},\,\,
f_2 = \underline{e} \cdot V \vectln{c}{0\\0\\1}
$$ 
(roughly speaking, $V\in \rM_{dd,3}(\sR)$ is the \emph{matrix of the filtration on} $\cM$).

Let $\mathrm{Mat}_{\underline{e},\underline{f}}(\varphi_2)\in\mathbf{GL}_{dd,3}(\sR)$ be the matrix such that
\begin{eqnarray*}
\varphi_2(f_0) = \underline{e}\cdot\mathrm{Mat}_{\underline{e},\underline{f}}(\varphi_2) \vectln{c}{1\\0 \\0},\,
\varphi_2(f_1) = \underline{e}\cdot\mathrm{Mat}_{\underline{e},\underline{f}}(\varphi_2) \vectln{c}{0\\1\\ 0},
\,
\varphi_2(f_2) = \underline{e}\cdot\mathrm{Mat}_{\underline{e},\underline{f}}(\varphi_2) \vectln{c}{0\\0\\1}.
\end{eqnarray*} 
(i.e. $\mathrm{Mat}_{\underline{e},\underline{f}}(\varphi_2)$ is the \emph{matrix of the Frobenius} with respect to $\underline{f}$).

We establish some preliminary lemmas to perform the induction argument to prove Theorem \ref{phifilr}. The first lemma lets us translate the effect of a change of basis on the matrices describing the Frobenius and the filtration.

From now on any basis $\un{e}=(e_0,e_1,e_2)$ and system of generators $\un{f}=(f_0,f_1,f_2)$ for $\Fil^2\cM/(\Fil^p\SBr\cdot \cM)$ are always understood to be compatible with the framing on $\tau$.

\begin{lem}
\label{lemmaprimo new}
Let  $A\in \mathrm{GL}_{dd,3}(\mathscr{R})$. Let  $\underline{e}'\defeq\underline{e}\cdot A$ be a new $\SBr$-basis for $\mathcal{M}$,  compatible with the framing on $\tau$. Assume that there exist $B\in \mathrm{GL}_{dd,3}(\sR)$ and $V'\in\mathrm{M}_{dd,3}(\sR)$ satisfying
$$
AV'=VB.
$$
The elements 
$$
f_0'\defeq \underline{e}'\cdot  V'\vectln{c}{1\\0\\0},\quad
f_1'\defeq \underline{e}'\cdot  V'\vectln{c}{0\\1\\0},
\quad
f'_2\defeq \underline{e}'\cdot V'\vectln{c}{0\\ 0\\1}
$$
form a system of $\SBr$-generators for $\mathrm{Fil}^r\mathcal{M}/\left(\Fil^p \SBr\cdot \mathcal{M}\right)$.  If we further assume that $A=\mathrm{Mat}_{\underline{e},\underline{f}}(\varphi_r)$, then we have
\begin{equation} 
\label{phichange new}
\mathrm{Mat}_{\underline{e}',\underline{f}'}(\varphi_2)=\varphi(B).
\end{equation}
\end{lem}
\begin{proof}
The elements $f'_0,f'_1,f'_2$ are obviously in $\mathrm{Fil}^2 \mathcal{M}$ and, since $B$ is invertible, they form a system of $\SBr$ generators for $\mathrm{Fil}^2 \mathcal{M}$ modulo $\Fil^p\SBr\cdot \cM$. 
The last statement is a simple translation of the definition of $\mathrm{Mat}_{\underline{e}',\underline{f}'}(\varphi_2)$, $\mathrm{Mat}_{\underline{e},\underline{f}}(\varphi_2)$, recalling that $\varphi_2$ is $\varphi$-semilinear.
\end{proof}

The previous Lemma will be widely used when $A,B\in\mathbf{T}_{3}(\mathcal{O}_E)+\rM_{dd,3}\left(p^n\sR\right)$, for $n\geq1$.
We now give a criterion for when we can find a $V'=V_{\un{e}',\un{f}'}$ which allows us to perform change of basis as in Lemma \ref{lemmaprimo new}:
\begin{lem}
\label{lemmasecondo new}
Let $V\in\rM_{dd,3}(\sR)$ be a matrix associated to the filtration $\Fil^2\cM$ on $\cM$. Then $V^{\mathrm{adj}}\in E(u)\rM_{dd,3}(\sR)$.
Moreover:
\begin{itemize}
	\item[$i)$] Let $A\in \rGL_{dd,3}(\sR)$ and assume there exists $V'\in\mathrm{M}_{dd,3}(\sR)$ such that
$$
\frac{1}{E(u)}V^{\mathrm{adj}}A V'\in E(u)^2\mathbf{GL}_{dd,3}(\sR)+\rM_{dd,3}(\cJ).
$$

Then, there exist $B\in \mathbf{GL}_{dd,3}(\sR)+\rM_{dd,3}(p,\Fil^1\sR)$ such that
$$
AV'=VB.
$$

	\item[$ii)$] Let $A\in \mathbf{T}_{3}(\cO_E)+\rM_{dd,3}\left(p^n\sR\right)$ and assume there exists $V'\in\mathrm{M}_{dd,3}(\sR)$ such that
$$
\frac{1}{E(u)}V^{\mathrm{adj}}A V'\in E(u)^2\rT_3(\cO_E)+\rM_{dd,3}(p^n\cJ).
$$

Then, there exist $B\in \rT_3(\cO_E)+\rM_{dd,3}(p^n(p,\Fil^1\sR))$ such that
$$
AV'=VB.
$$
\end{itemize}
\end{lem}
\begin{proof}
The first assertion is clear from the height condition. 

We prove $i)$, the proof of $ii)$ being identical.
Let us write
$$
\frac{1}{E(u)}V^{\mathrm{adj}}A V'= E(u)^2(B+N_0)
$$
where $B\in \mathbf{GL}_{dd,3}(\mathscr{R})$ and $N_0\in \mathrm{M}_{dd,3}\left((p,\mathrm{Fil}^1\SBr)\right)$.

We obtain
$$
E(u)^2 V\cdot (B+N_0)=\frac{1}{E(u)}(V\cdot V^{\mathrm{adj}})A V'=E(u)^2 A V'
$$
and the statement follows.
\end{proof}

The proof of Theorem \ref{phifilr} is now a $p$-adic approximation argument, in which we distinguish two steps.
The first one (Proposition \ref{even new}) shows that at the $2n$-th step of the $p$-convergence argument we are able to make the $\varphi_r$-action \emph{lower triangular} modulo $p^{n+1}$.

\begin{prop}[\emph{Even case}]
\label{even new}
Let $A_{0}\in \mathbf{GL}_{dd,3}(\sR)$ and, for $n\geq1$, let $A_{2n}\in \mathbf{T}_{3}(\cO_E)+\mathrm{M}_{dd,3}\left(p^n\sR\right)$.

Let 
$$
V^{(2n)}=\maqn{ccc}{1&0&0\\u^{\left[ a_1-a_0\right]}v_{1,0}^{(2n)}&1&0\\
u^{\left[ a_2-a_0\right]}v_{2,0}^{(2n)}&u^{\left[ a_2-a_1\right]}v_{2,1}^{(2n)}&1}\mathrm{Diag}(1,E(u),E(u)^2)
$$ 
for some elements $v^{(2n)}_{1,0},\,v^{(2n)}_{2,1}\in\cO_E$, $v^{(2n)}_{2,0}\in\cO_E\oplus E(u)\cO_E$.

Then there exist elements $v^{(2n+1)}_{1,0},\,v^{(2n+1)}_{2,1}\in\cO_E$, $v^{(2n+1)}_{2,0}\in\cO_E\oplus E(u)\cO_E$ such that:
\begin{eqnarray}
\label{equation step 2n}\\\nonumber
\frac{1}{E(u)}V^{(2n),\mathrm{adj}} A_{(2n)} V^{(2n+1)}\in\left\{
\begin{array}{cc}
E(u)^2\overline{\mathbf{B}}_{dd,3}(\mathscr{R})+\rM_{dd,3}\left(\cJ\right)&\text{if}\,\,n=0\\
\\
E(u)^2\left(\mathbf{T}_{3}(\mathcal{O}_E)+\rL_{dd,3}\left(p^n\mathscr{R}\right)\right)
+\rM_{dd,3}\left(p^n\cJ\right)&\text{if}\,\,n\geq 1\\
\end{array}
\right.
\end{eqnarray}
where 
$$
V^{(2n+1)}\defeq
\maqn{ccc}{1&0&0\\u^{\left[ a_1-a_0\right]}v_{1,0}^{(2n+1)}&1&0\\
u^{\left[ a_2-a_0\right]}v_{2,0}^{(2n+1)}&u^{\left[ a_2-a_1\right]}v_{2,1}^{(2n+1)}&1}\mathrm{Diag}(1,E(u),E(u)^2).
$$

In particular, any element $B_{2n}\in\rGL_{dd,3}(\sR)+\rM_{dd,3}\left((p,\Fil^1\sR)\right)$ deduced from the equation (\ref{equation step 2n}) via Lemma \ref{lemmasecondo new} verifies:
\begin{itemize}
	\item[(i)] $B_{2n}\in\rB^{\opp}_{dd,3}(\sR)+\rM_{dd,3}((p,\Fil^1\sR))$ if $n=0$;
	\item[(ii)] $B_{2n}\in\rT_3(\cO_E)+\rL_{dd,3}(p^n\sR)+\rM_{dd,3}\left(p^{n}(p,\Fil^1\sR)\right)$ if $n>0$.
\end{itemize} 
\end{prop}

In a similar fashion, we can \emph{diagonalize} the $\varphi_r$-action modulo $p^{n+1}$ at the step $2n+1$:
\begin{prop}[Odd case]
\label{odd new}
Let $n\geq 0$ and let
$$
A_{2n+1}\in\mathbf{T}_{3}(\mathcal{O}_E)+\rM_{dd,3}\left((p^n\cJ+p^n\cI)\right)+\rM_{dd,3}(p^{n+1}\sR).
$$
Let 
$$
V^{(2n+1)}=\maqn{ccc}{1&0&0\\u^{\left[ a_1-a_0\right]}v_{1,0}^{(2n+1)}&1&0\\
u^{\left[ a_2-a_0\right]}v_{2,0}^{(2n+1)}&u^{\left[ a_2-a_1\right]}v_{2,1}^{(2n+1)}&1}\mathrm{Diag}(1,E(u),E(u)^2)
$$ 
for some elements $v^{(2n+1)}_{1,0},\,v^{(2n+1)}_{2,1}\in\cO_E$, $v^{(2n+1)}_{2,0}\in\cO_E\oplus E(u)\cO_E$.

Then there exist elements $v^{(2n+2)}_{1,0},\,v^{(2n+2)}_{2,1}\in\cO_E$, $v^{(2n+2)}_{2,0}\in\cO_E\oplus E(u)\cO_E$ such that
\begin{eqnarray}
\label{equation step 2n+1}\\\nonumber
\frac{1}{E(u)}V^{(2n+1),\mathrm{adj}} A_{(2n+1)} V^{(2n+2)}\in
E(u)^2\mathbf{T}_{3}(\mathcal{O}_E)+\rM_{dd,3}\left(p^n\cJ\right)
\end{eqnarray}
where 
$$
V^{(2n+2)}\defeq 
\maqn{ccc}{1&0&0\\u^{\left[ a_1-a_0\right]}v_{1,0}^{(2n+2)}&1&0\\
u^{\left[ a_2-a_0\right]}v_{2,0}^{(2n+2)}&u^{\left[ a_2-a_1\right]}v_{2,1}^{(2n+2)}&1}\mathrm{Diag}(1,E(u),E(u)^2).
$$

In particular, any element $B_{2n+1}\rGL_{dd,3}(\sR)+\rM_{dd,3}\left((p,\Fil^1\sR)\right)$ deduced from the equation (\ref{equation step 2n+1}) via Lemma \ref{lemmasecondo new} verifies:
$$
B_{2n+1}\in\rT_3(\cO_E)+\rM_{dd,3}\left(p^{n}(p,\Fil^1\sR)\right).
$$

\end{prop}

The proof of Proposition \ref{even new}, \ref{odd new}, which is the key technical part in the approximation argument, is carried out in section \S \ref{subsectionproof new} below.

We now show how Proposition \ref{even new} and \ref{odd new} let us perform the $p$-adic convergence argument giving rise to Theorem \ref{phifilr}.

\begin{lem}[\emph{even} $\Rightarrow$ \emph{odd}]
Let $n\in\N$ and write
$$
\frac{1}{E(u)}V^{(2n),\mathrm{adj}} A_{(2n)} V^{(2n+1)}=E(u)^2\left(B_{2n}\right)
$$
where $V^{(2n)}$, $V^{(2n+1)}$, $A_{(2n)}$ and $B_{2n}$ are as in the statement of Proposition \ref{even new}.

Then the element $A_{2n+1}\defeq\varphi\left(B_{2n}\right)$ veryfies the hypotheses of Proposition \ref{odd new}. 
\end{lem}
\begin{proof}
It is an immediate consequence of Lemma \ref{lemphir new new} $i)$.
\end{proof}

\begin{lem}[\emph{odd} $\Rightarrow$ \emph{even}]
Let $n\in\N$ and write
$$
\frac{1}{E(u)}V^{(2n+1),\mathrm{adj}} A_{(2n+1)} V^{(2n+2)}=E(u)^2B_{2n+1}
$$
where $V^{(2n+1)}$, $V^{(2n+2)}$, $A_{(2n+1)}$ and $B_{2n+1}$ are as in the statement of Proposition \ref{odd new}.

Then the element $A_{2n+2}\defeq\varphi\big(B_{2n+1})$ veryfies the hypotheses of Proposition \ref{even new}. 
\end{lem}
\begin{proof}
It is an immediate consequence of Lemma \ref{lemphir new new} $ii)$.
\end{proof}

\subsection{Proofs of Propositions $\ref{even new}$, $\ref{odd new}$}
\label{subsectionproof new}

\paragraph{\textbf{The even case.}}

In order to lighten notations we write $A\defeq A_{2n}$, $V^{(2n)}=V\cdot\mathrm{Diag}(1,E(u),E(u)^2)$, 
$V^{(2n+1)}=V'\cdot\mathrm{Diag}(1,E(u),E(u)^2)$ for appropriate $V,\,V'\in\rU^{\opp}_{dd,3}(\sR)$ and define $v_{i,j}\defeq \left(V'\right)_{i,j}$.

We therefore have
$$
\frac{1}{E(u)}\cdot V^{(2n),\mathrm{adj}}=\mathrm{Diag}(E(u)^2,E(u),1)V^{\mathrm{adj}}
$$
and $W\defeq V^{\mathrm{adj}}=(u^{\left[ a_i-a_j\right]}w_{i,j})\in\rU^{\opp}_{dd,3}(\sR)$.

An immediate computation gives:
\begin{eqnarray}
\label{eqnarrayeven new}\\
\nonumber
\underbrace{\mathrm{Diag}(E(u)^2,E(u),1)\cdot W A V'\cdot \mathrm{Diag}(1,E(u),E(u)^2)}_{
\frac{1}{E(u)}V^{(2n),\mathrm{adj}} A_{(2n)} V^{(2n+1)}
}
\in\rL_{dd,3}(\mathscr{R})+\rM_{dd,3}\left((p^n\cJ)\right).
\end{eqnarray}
so that, using Lemma \ref{lemconto3.5 new} and $A\in\rGL_3(\sR)$ we see that $(WAV')_{i,i}\in A_{i,i}+p^n(p,\mathrm{Fil}^1\sR)$.
Hence the diagonal entries in the LHS of (\ref{eqnarrayeven new}) lie actually in  $E(u)^2\cO_E^{\times}+p^n\cJ$.

In order to establish Proposition \ref{even new} we are therefore left to find, for $i>j$, elements $v_{i,j}\in\mathscr{R}$ in such a way that 
\begin{eqnarray}
\label{eqnarrayeven1 new}\\
\nonumber
\left(\mathrm{Diag}(E(u)^2,E(u),1)\cdot W A V'\cdot \mathrm{Diag}(1,E(u),E(u)^2)\right)_{i,j}\equiv 0\,\, \mathrm{modulo}\, u^{[a_i-a_j]}\cdot \left(p^n\mathrm{Fil}^{2}\sR\right)
\end{eqnarray}
that is to say
$$
\big(W A V'\big)_{i,j}\equiv 0\,\, \mathrm{modulo}\, u^{[a_i-a_j]}\cdot \left(p^n\mathrm{Fil}^{i-j}\sR\right).
$$
for $2\geq i>j\geq 0$.

Let us write $WA=(u^{\left[ a_i-a_j\right]}m_{i,j})_{i,j}\in\mathbf{GL}_{dd,3}(\mathscr{R})$ where $m_{i,j}\in\sR$.
Then the condition (\ref{eqnarrayeven1 new}) is equivalent to following systems of linear equations:
\begin{equation}
\label{eqnarrayeven3 new}
D_0\underbrace{\maqn{cc}{m_{1,1}&m_{1,2}\\u^em_{2,1}&m_{2,2}}}_{M_0}\vectln{c}{v_{1,0}\\v_{2,0}}\equiv -D_0\vectln{c}{m_{1,0}\\m_{2,0}}\qquad\text{modulo}\quad
p^n\mathrm{Fil}^{2}\sR
\end{equation}
(where $D_0\defeq \mathrm{Diag}\left(E(u)u^{\left[ a_1-a_0\right]},u^{\left[ a_2-a_0\right]}\right)$) and
\begin{equation}
\label{eqnarrayeven3.1 new}
E(u)u^{\left[ a_2-a_1\right]}{m_{2,2}}\cdot v_{2,1}\equiv -E(u)u^{\left[ a_2-a_1\right]} m_{2,1}\qquad\text{modulo}\,
p^n\mathrm{Fil}^{2}\sR
\end{equation}
We have $m_{i,i}\in \sR^{\times}$ so that $M_0$, $m_{2,2}$ are invertible and we are left to define
\begin{eqnarray*}
\vectln{c}{v_{1,0}\\v_{2,0}}\defeq - M_{0}^{-1}\cdot\vectln{c}{m_{1,0}\\m_{2,0}},&\qquad&v_{2,1}\defeq -m_{2,2}^{-1}m_{2,1}.
\end{eqnarray*}
Note that the elements $v_{1,0},\,v_{2,1}$ can be assumed to be in $\cO_E$ and $v_{2,0}$ can be assumed to be in $\cO_E\oplus E(u)\cO_E$.
This concludes the proof of Proposition \ref{even new}.

\bigskip

\paragraph{\textbf{The odd case.}}
As for the previous paragraph, in order to lighten notations we write $A\defeq A_{2n+1}$ and 
$$
V^{(2n+1)}\defeq V\mathrm{Diag}(1,E(u),E(u)^2),\,\,V^{(2n+2)}=V'\mathrm{Diag}(1,E(u),E(u)^2)
$$
for appropriate $V,\,V'\in\rU^{\opp}_{dd,3}(\sR)$ and $v_{i,j}\defeq \left(V'\right)_{i,j}$.

We therefore have
$$
\frac{1}{E(u)}V^{(2n+1),\mathrm{adj}}=\mathrm{Diag}(E(u)^2,E(u),1)V^{\mathrm{adj}}
$$
and $W\defeq V^{\mathrm{adj}}=(u^{\left[ a_i-a_j\right]}w_{i,j})\in{\rU}^{\opp}_{dd,3}(\sR)$ for some $w_{i,j}\in \sR$.

Since $A\in\mathbf{T}_{3}(\mathcal{O}_E)+\rM_{dd,3}\left(p^n\sR\right)$, we see that $\left(W\cdot A\right)_{i,j}\equiv 0$ modulo $p^n$ whenever $j>i$, hence
\begin{eqnarray}
\label{eqnarrayodd new}\\
\nonumber
\underbrace{\mathrm{Diag}(E(u)^2,E(u),1)W  A  V'\mathrm{Diag}(1,E(u),E(u)^2)}_{
\frac{1}{E(u)}V^{(2n+1),\mathrm{adj}} A_{(2n+1)} V^{(2n+2)}
}\in\rL_{dd,3}(\sR)+\rM_{dd,3}\left(p^n\cJ\right).
\end{eqnarray}
Again, an immediate manipulation (using Lemma \ref{lemconto3.5 new} and the fact that $A\in\mathbf{T}_{3}(\mathcal{O}_E)+\rM_{dd,3}\left(p^n\sR\right)$) shows that 
$$
(W A V')_{i,i}\in A_{i,i}+(p^n(p,\mathrm{Fil}^1\sR))
$$ 
so that the diagonal entries in the LHS of (\ref{eqnarrayodd new}) lie actually in  $E(u)^2\cO_E^{\times}+p^n\cJ$.

In order to establish Proposition \ref{odd new} we are left to determine, for $2\geq i>j\geq 0$, the elements $v_{i,j}\in\mathscr{R}$ in such a way that 
\begin{eqnarray}
\label{eqnarrayodd1 new}\\
\nonumber
\bigg(\mathrm{Diag}(E(u)^2,E(u),1)W  A  V'\mathrm{Diag}(1,E(u),E(u)^2)\bigg)_{i,j}\equiv 0\,\,\mathrm{modulo}\,\, u^{[a_i-a_j]}\cdot \left(p^n\cJ\right),
\end{eqnarray}
that is to say
\begin{eqnarray*}
(W A V')_{i,j}\in u^{[a_i-a_j]}\cdot \left(p^n\left(p\Fil^{i-j}\sR,\Fil^{i-j+1}\sR\right)\right)
\end{eqnarray*}
(for $2\geq i>j\geq 0$) with \emph{the additional requirements} that 
\begin{eqnarray}
\label{extra}
v_{1,0},\,v_{2,1}\in \cO_E,\qquad&&v_{2,0}\in \cO_E\oplus E(u)\cO_E.
\end{eqnarray}

Let us write $WA=(u^{\left[ a_i-a_j\right]}m_{i,j})_{i,j}\in\mathbf{GL}_{dd,3}(\mathscr{R})$.
Then we have
\begin{equation}
\label{eqnarrayodd2 new}
m_{i,j}\in w_{i,j}A_{j,j}+p^n\cJ+p^n\cI+p^{n+1}\sR
\end{equation}
for all $0\leq i,j\leq 2$, as $A\in\mathbf{T}_{3}(\mathcal{O}_E)+\rM_{dd,3}\left(p^n\cJ+p^n\cI+p^{n+1}\sR\right)$. In particular, $m_{j,j}\in\sR^{\times}$.

Moreover, by hypotheses we have $V_{1,0},\,V_{2,1}\in\cO_E$ and $V_{2,0}\in \cO_E\oplus E(u)\cO_E$ so that, using Lemma \ref{lem3.6 new} we obtain:
\begin{eqnarray}
\label{eqnarrayodd3 new}
w_{2,1},\,w_{1,0}\in \cO_E,\,\,w_{2,0}\in \cO_E\oplus E(u)\cO_E.
\end{eqnarray}

As we did in the previous paragraph (the proof of the even case \ref{even new}) we are left consider the following systems of linear equations
\begin{eqnarray}
\label{eqnarrayodd5 new}\\
\nonumber
D_0\underbrace{\maqn{cc}{m_{1,1}&m_{1,2}\\u^em_{2,1}&m_{2,2}}}_{M_0}\vectln{c}{v_{1,0}\\v_{2,0}}\equiv -D_0\vectln{c}{m_{1,0}\\m_{2,0}}\qquad\text{modulo}\quad
p^n\cJ
\end{eqnarray}
where $D_0\defeq \mathrm{Diag}(E(u)u^{\left[ a_1-a_0\right]},u^{\left[ a_2-a_0\right]})$ and
\begin{eqnarray}
\label{eqnarrayodd5.1 new}\\
\nonumber
E(u)u^{\left[ a_2-a_1\right]}m_{2,2}v_{2,1}\equiv -E(u)u^{\left[ a_2-a_1\right]}m_{2,1}\qquad\text{modulo}\quad
p^n\cJ.
\end{eqnarray}

Then condition (\ref{eqnarrayodd1 new}) is now translated into the existence of solutions to (\ref{eqnarrayodd5 new}), (\ref{eqnarrayodd5.1 new}), where $v_{i,j}$ verify moreover (\ref{extra}).

By (\ref{eqnarrayodd2 new}) and (\ref{eqnarrayodd3 new}) we have $m_{2,2}\in \cO_E^{\times}+p^n\cJ+p^n\cI+p^{n+1}\sR$ and
$$
M_{0}\in 
\maq{\cO_E^{\times}}{0}{\left(\cO_E\oplus E(u)\cO_E\right)}{\cO_E^{\times}}+\rM_2(p^n\cJ+p^n\cI+p^{n+1}\sR);
$$

and hence, by Lemma \ref{lem3.6 new} we have
\begin{equation}
\label{inverse M0 new}
M_{0}^{-1}\in \maq{\cO_E^{\times}}{0}{\left(\cO_E\oplus E(u)\cO_E\right)}{\cO_E^{\times}}+
\rM_2(p^n\cJ+p^n\cI+p^{n+1}\sR).
\end{equation}

If we define $v_{2,1}\defeq -m_{2,2}m_{2,1}$ and 
\begin{eqnarray*}
\vectln{c}{v_{1,0}\\v_{2,0}}\defeq 
-M_0^{-1}\cdot
\vectln{c}{m_{1,0}\\m_{2,0}}
\end{eqnarray*}
we deduce from (\ref{inverse M0 new}), $(\ref{eqnarrayodd3 new})$ (and the important fact that $p^{n+1}\sR\in p^{n+1}\cO_E+p^n\cI+p^n\cJ$, cf. Lemma \ref{lemcalcul1 new}) that  
\begin{equation*}
v_{1,0},\,v_{2,1}\in \cO_E+\left(p^n\cJ,p^n\cI\right),\qquad v_{2,0}\in \cO_E\oplus E(u)\cO_E+\left(p^n\cJ,p^n\cI\right).
\end{equation*}

Hence 
$$
v_{1,0},\,v_{2,1}\in \cO_E+p^{n+1} E(u)\cO_E+p^n\cJ\qquad v_{2,0}\in \cO_E\oplus E(u)\cO_E+p^n\cJ
$$ 
and since $E(u)\cdot\left(p^{n+1} E(u)\cO_E\right)\in p^n\cJ$ we deduce that the elements $v_{i,j}$ defined this way verify condition (\ref{eqnarrayodd5 new}) \emph{and} can be assumed, without loss of generality, to verify condition (\ref{extra}).

This ends the proof of Proposition \ref{odd new}.

\section{Monodromy on Breuil modules}

The aim of this section is to give necessary and sufficient conditions on ``ordinary'' quasi-Breuil modules with descent data to admit a monodromy operator.

More precisely, we consider quasi-Breuil modules $\cM$ over $\overline{R}$ with descent data of principal series type $\tau=\teich{\omega}^{a_0}\oplus\teich{\omega}^{a_1}\oplus\teich{\omega}^{a_2}$, where the integers $a_i$ verify the strongly genericity assumption (\ref{strong genericity}) which arise as reductions modulo $p$ of families of the form given in Theorem \ref{phifilr}. The main result is Proposition \ref{essential image}. The proof is a fairly direct computation with the matrix for $N$ relying crucially on the genericity condition on the descent data.  

From Proposition \ref{bijection ordrer preserving} we have a lattice of quasi-Breuil submodules of $\barcM$, whose constituents $\overline{\cM}_{i,j}$ are characterized by $\Tqst^{2}(\overline{\cM}_{i,j})\cong \rhobar_{i,j}\vert_{G_{(\Qp)_{\infty}}}$. In particular, we have a filtration with rank one quotients:
\begin{equation}
\label{filtration Breuil submodules}
0\subseteq \barcM_2\subseteq\barcM_{2,1}\subseteq \barcM.
\end{equation}

We can describe quasi-Breuil modules $\overline{\cM}$ (with principal series type $\tau$)  such that $\Tqst^{2}(\overline{\cM})\cong\rhobar\vert_{G_{(\Qp)_{\infty}}}$:

\begin{prop}
\label{diagonalize mod p}
Let $\barcM\in\FBrModddN[2]$ be a quasi-Breuil module such that $\Tst^{2}(\barcM)\cong \rhobar\vert_{G_{(\Qp)_{\infty}}}$, where $\rhobar$ is ordinary Fontaine-Laffaille and strongly generic. Assume that $\barcM$ has descent data of type $\tau$.

There exists a framed basis $\underline{e}=(e_0,e_1,e_2)$ and a framed system of generators $\underline{f}=(f_0,f_1,f_2)$ for $\Fil^2\barcM$ such that:
\begin{eqnarray}
\label{shape ordinary}\\
\Mat_{\un{e}}(\Fil^2\barcM)=\small\begin{pmatrix}
1&0&0\\
u^{[a_1-a_0]}v_{1,0}&u^e&0\\
u^{[a_2-a_0]}(v_{2,0}+u^ev'_{2,0})&u^{e+[a_2-a_1]}v_{2,1}&u^{2e}
\end{pmatrix},
&
\Mat_{\un{e},\un{f}}(\phz_2)=\begin{pmatrix}
\alpha_0&0&0\\
0&\alpha_1&0\\
0&0&\alpha_2
\end{pmatrix}&
\nonumber
\end{eqnarray}
where $v_{i,j}, v'_{2,0}\in\F$, $\alpha_i\in\F^{\times}$.
\end{prop} 
\begin{proof}
In what follows, we write $\barS_0$ to denote the $\omega^0$-isotypical component of $\barS$ (i.e. $\barS_0=\F[u^e]/u^{ep}$).

As in the proof of Proposition \ref{corofiltrazione} we see that there exists a basis $\underline{e}=(e_0,e_1,e_2)$ and a system of generators $\underline{f}=(f_0,f_1,f_2)$, compatible with both the framing on $\tau$ and the filtration (\ref{filtration Breuil submodules}), and such that
\begin{equation*}
V\defeq\Mat_{\un{e}}([f_0,f_1,f_2])=\begin{pmatrix}
1&0&0\\
u^{[a_1-a_0]}v_{1,0}&u^e&0\\
u^{[a_2-a_0]}v_{2,0}&u^{e+[a_2-a_1]}v_{2,1}&u^{2e}
\end{pmatrix}
\end{equation*}
where $v_{1,0},\,v_{2,1}\in\F$ and $v_{2,0}\in \F\oplus u^e\F$. 
As $\underline{e},\,\,\underline{f}$ are compatible with the filtration (\ref{filtration Breuil submodules}) we moreover deduce that $A\defeq\mathrm{Mat}_{\underline{e},\underline{f}}(\phz_2)\in \mathbf{GL}_{dd,3}\left(\barS_0\right)$ is actually in ${\rB}^{\opp}_{3,dd}\left(\barS_0\right)$, the Borel of lower triangular matrices.

We now apply Lemma \ref{lemmaprimo new}: it is easy to see that there exists matrices $V'\in\rL_{dd,3}\left(\barS_0\right)$ and $B\in {\rB}^{\opp}_{dd,3}\left(\barS_0\right)$ such that
$$
AV'\in VB+u^{3e}\rM_{dd,3}\left(\barS_0\right)
$$
and therefore $A'\defeq \varphi(B)$ is the matrix of the Frobenius action on $\cM$ with respect to the basis $\underline{e}'\defeq \underline{e}\cdot A$ and the system of generators $\underline{f}'\defeq \underline{e}V'$. Note that, as $A$, $V'$ are lower triangular, the new basis $\underline{e}'$ and new system of generators $\underline{f}'$ are compatible with the filtration (\ref{filtration Breuil submodules}).

We can now repeat the argument: at the end of the second iteration we end up with a framed basis for $\barcM$ and a framed system of generators for $\Fil^2\barcM$ which are compatible with both the framing on $\tau$ and the filtration (\ref{filtration Breuil submodules}), as in the statement.
\end{proof}

Let $\overline{R}$ be a complete local Noetherian $\F$-algebra with residue field $\F$. 

\begin{defi} \label{ordform}Let  $\cM$ be a quasi-Breuil module  over $\overline{R}$ with descent datum.  We say that a framed basis $\underline{e}$ and a framed system of generators $\underline{f}$ for $\Fil^2(\cM)$ is in \emph{ordinary form} if 
\begin{eqnarray*}
\Mat_{\un{e}}(\Fil^2\cM)=\begin{pmatrix}
1&0&0\\
u^{[a_1-a_0]}v_{1,0}&u^e&0\\
u^{[a_2-a_0]}(v_{2,0}+u^ev'_{2,0})&u^{e+[a_2-a_1]}v_{2,1}&u^{2e}
\end{pmatrix}
\end{eqnarray*}
and 
\begin{eqnarray*}
\Mat_{\un{e},\un{f}}(\phz_2)=\begin{pmatrix}
\alpha_0&0&0\\
0&\alpha_1&0\\
0&0&\alpha_2
\end{pmatrix}
\nonumber
\end{eqnarray*}
where $v_{i,j}, v'_{2,0}\in \overline{R}$, $\alpha_i\in\overline{R}^{\times}$.
\end{defi}

The above definition is closely related to the notion of gauge basis (\ref{def gauge}). The main difference being that here we are specifying both $\underline{e}$ and $\underline{f}$. We are now in the position of state the main result of this section.

\begin{prop}
\label{essential image}
Let $\cM\in\RBrModddN[2]$ be a quasi-Breuil module with a framed basis $\underline{e}=(e_0,e_1,e_2)$ and a framed system of generators $\underline{f}=(f_0,f_1,f_2)$ for $\Fil^2 \cM$ in ordinary form $($\ref{ordform}$)$.  Assume $a_0, a_1, a_2$ satisfy the strong genericity hypothesis $($\ref{strong genericity}$)$. Then $\cM$ is endowed with a monodromy operator if and only if $v_{2,0}=0$. 

In this case, one has
$$
N=\begin{pmatrix}
0&0&0\\
u^{[a_1-a_0]}P_{1,0}(u^e)&0&0\\
u^{[a_2-a_0]}P_{2,0}(u^e)&u^{[a_2-a_1]}P_{2,1}(u^e)&0
\end{pmatrix}
$$
where $P_{i+1,i}=\alpha_{i+1}\alpha_{i}^{-1}[a_{i+1}-a_{i}]v_{i+1,i}u^{e[a_{i+2}-a_{i}]}$ for $i=0,1$ and
$$
P_{2,0}=-\alpha_{2}\alpha_0^{-1}(v'_{2,0}([a_2-a_0]-1)-v_{1,0}v_{1,2}[a_1-a_0])u^{e\left[a_2-a_0\right]}.
$$
\end{prop} 

The rest of this section is devoted to the proof of Proposition \ref{essential image}. From now on, we fix a quasi-Breuil module $\cM$ over $\overline{R}$ with descent data $\tau$ in ordinary form. To lighten notations, we write $\barS=\barS_R=\overline{R}[u]/(u^{ep})$ and, as in the proof of Proposition \ref{diagonalize mod p}, we write $\barS_0$ to denote the $\omega^0$-isotypical component of $\barS$.
 As the monodromy operator is compatible with the descent data, we deduce from \cite{EGH}, Lemma 3.3.2 that the action of a monodromy $N$ on $\cM$ is described by
$$
N=\begin{pmatrix}
0&0&0\\
u^{[a_1-a_0]}P_{1,0}(u^e)&0&0\\
u^{[a_2-a_0]}P_{2,0}(u^e)&u^{[a_2-a_1]}P_{2,1}(u^e)&0
\end{pmatrix}
$$
in the basis $\un{e}$ for some $P_{i,j}(u^e)\in \barS_0$.  

Recall that the monodromy operator $N$ satisfies
\begin{itemize}
	\item[$i)$] $u^e N(f_{i})\in \Fil^2\cM$;
	\item[$ii)$] $\phz_2(u^eN(f_{i}))=N(\phz_2(f_{i}))$
\end{itemize}
and the usual Leibnitz relation
\begin{equation}
\label{leibenitz}
N(Q(u)\cdot x)=-u\frac{\partial}{\partial u}(Q(u))x+ Q(u) N(x)
\end{equation}
for any $Q(u)\in\barS$, $x\in \cM$. We will write $N_{\barS}$ to denote the monodromy $-u\frac{\partial}{\partial u}$ on $\barS$.

\begin{proof}[Proof of Propostion \ref{essential image}]

Since $\cM$ is in ordinary form, we can define the subquotients in $\cM_{i,j}$ as we did for $\barcM_{i,j}$. Let us first consider $\cM_{2,1}$. A simple computation, using the Leibniz relation (\ref{leibenitz}) gives 
\begin{eqnarray*}
N(f_1)=-ef_1+u^{e+[a_2-a_1]}\left(\vectln{lll}{0\\0\\P_{2,1}(u^e)}-[a_2-a_1]v_{2,1}\vectln{lll}{0\\0\\1}\right)
\end{eqnarray*}
hence $u^eN(f_1)=u^ef_1+u^{[a_2-a_1]}(P_{2,1}(u^e)-[a_2-a_1]v_{2,1}) f_2 \in \Fil^2\cM$.

Using (\ref{leibenitz}) and noticing that $\phz(Q(u^e))=Q(0)$ for any $Q(u^e)\in \barS_0$, we further obtain
\begin{eqnarray*}
\phz_2(u^eN(f_1))&=&u^{[a_2-a_1]}u^{e[a_2-a_1]}\alpha_2(P_{2,1}(0)-[a_2-a_1]v_{2,1})e_2,
\end{eqnarray*}
and as $\cM$ is ordinary we have
\begin{eqnarray*}
N(\phz_2(f_1))&=&\alpha_1u^{[a_2-a_1]}P_{2,1}(u^e)e_2.
\end{eqnarray*}
Hence, from $ii)$ we deduce $P_{2,1}(u^e)=-\alpha_2\alpha_1^{-1}[a_2-a_1]v_{2,1}u^{e[a_2-a_1]}$. By a similar argument with $\cM_{1, 0},$ we deduce that $P_{1,0}(u^e)=-\alpha_1\alpha_0^{-1}v_{1,0}[a_1-a_0]u^{e[a_1-a_0]}$.  Note that both $P_{2,1}(u^e)$ and $P_{1,0}(u^e)$ are in $\Fil^2 S$ by the genericity condition. 

We now consider $N(f_0)$. A laborious manipulation but without difficulties, using the Leibniz relation and the definition of $f_1$ provides us with the following:
\begin{eqnarray*}
u^eN(f_0)&=&u^e\left(\vectln{lll}{0\\u^{[a_1-a_0]}P_{1,0}(u^e)\\u^{[a_2-a_0]}P_{2,0}(u^e)}+
N_{\barS}(u^{[a_1-a_0]}v_{1,0})\vectln{lll}{0\\1\\0}+\right.\\
&&\left.\qquad\qquad+ u^{[a_1-a_0]}v_{1,0}\vectln{lll}{0\\0\\u^{[a_2-a_1]}P_{2,1}(u^e)}+N_{\barS}(u^{[a_2-a_0]}(v_{2,0}+u^ev'_{2,0}))\vectln{lll}{0\\0\\1}\right)\\
&\in&u^{[a_1-a_0]}\left(P_{1,0}(u^e)-[a_1-a_0]v_{1,0}\right)f_1+
u^{[a_2-a_0]}(1-[a_2-a_0])v'_{2,0}f_2+\\
&&
+u^e\left(u^{[a_2-a_1]+[a_1-a_0]}\left(-v_{2,1}P_{1,0}(u^e)+v_{1,0}v_{2,1}[a_1-a_0]\right)\right.\\
&&\qquad\qquad\qquad\qquad\qquad\left.
+u^{[a_2-a_0]}\left(P_{2,0}(u^e)-v_{2,0}[a_2-a_0]\right)\right)e_2+u^e\Fil^2\cM
\end{eqnarray*}
(where we used that $u^eP_{2,1}(u^e)e_2\in u^e\Fil^2\cM$).

Since we assume $e>a_2-a_1,a_1-a_0>0$ we have $[a_2-a_1]-[a_1-a_0]=e+[a_2-a_0]$. Therefore
\begin{eqnarray*}
u^eN(f_0)&\in&u^{[a_1-a_0]}\left(P_{1,0}(u^e)-[a_1-a_0]v_{1,0}\right)f_1+\\
&&\, +u^{[a_2-a_0]}\left(-v_{2,1}P_{1,0}(u^e)+v_{1,0}v_{2,1}[a_1-a_0]+(1-[a_2-a_0])v'_{2,0}\right)f_2+\\
&&\,\, 
+u^eu^{[a_2-a_0]}\left(P_{2,0}(u^e)-v_{2,0}[a_2-a_0]\right)e_2+u^e\Fil^2\cM
\end{eqnarray*}
so that Griffiths' transversality is verified if and only if
\begin{equation}
\label{condition griffiths}
P_{2,0}(u^e)-v_{2,0}[a_2-a_0]\in u^e\barS_0.
\end{equation} 

Let us write $P_{2,0}(u^e)=v_{2,0}[a_2-a_0]+u^e\widetilde{P}_{2,0}(u^e)$ for some $\widetilde{P}_{2,0}(u^e)\in \barS_0$. We now have 
\begin{eqnarray*}
u^eN(f_0)&\in&u^{[a_1-a_0]}\left(P_{1,0}(u^e)-[a_1-a_0]v_{1,0}\right)f_1+\\
&&\,+u^{[a_2-a_0]}\left(-v_{2,1}P_{1,0}(u^e)+v_{1,0}v_{2,1}[a_1-a_0]+(1-[a_2-a_0])v'_{2,0}+\widetilde{P}_{2,0}(u^e)\right)f_2 \\
&&\,\, +u^e\Fil^2\cM
\end{eqnarray*}
hence, imposing condition $ii)$, we deduce that:
\begin{eqnarray*}
\alpha_0P_{2,0}(u^e)&=&u^{e[a_2-a_0]}\alpha_2\left(-v_{2,1}P_{1,0}(0)+v_{1,0}v_{2,1}[a_1-a_0]+(1-[a_2-a_0])v'_{2,0}+\widetilde{P}_{2,0}(0)\right).
\end{eqnarray*}
In particular, $P_{2,0}(u^e)\in u^{2e}\barS_0$ by the genericity assumption. Returning to (\ref{condition griffiths}), we conclude that
\begin{eqnarray*}
\left[a_2-a_0\right]v_{2,0}&=&0;\\
P_{2,0}(u^e)&=&-\alpha_2\alpha_0^{-1}\left(v'_{2,0}([a_2-a_0]-1)-v_{1,0}v_{2,1}[a_1-a_0]\right)u^{e[a_2-a_0]}.
\end{eqnarray*} 

This proves the only if direction. 

For the converse, it is enough to notice that, by virtue of the previous computations, 
a monodromy $N:\cM\rightarrow \cM$ defined by 
\begin{eqnarray*}
N(e_0)&\defeq& -\alpha_1\alpha_0^{-1}v_{1,0}[a_1-a_0]u^{e[a_1-a_0]}e_1+\\
&&
\,\quad+\alpha_2\alpha_0^{-1}\left(v_{1,0}v_{2,1}[a_1-a_0]-v'_{2,0}([a_2-a_0]-1)\right)u^{e[a_2-a_0]}e_2,\\
&&
\\
N(e_1)&\defeq& -\alpha_2\alpha_1^{-1}v_{2,1}[a_2-a_1]u^{e[a_2-a_1]}e_2,\\
N(e_2)&\defeq&0
\end{eqnarray*}
verifies both conditions $i)$ and $ii)$, and the quasi-Breuil module $\cM$ is thus equipped with the structure of a Breuil module.
\end{proof}

\begin{rmk}
Following the same technique, it is possible to determine the monodromy operator even when the type $\tau$ lies in the \emph{upper alcove}. More precisely, if $\tau=\omega^{a_2}\oplus\omega^{a_1}\oplus\omega^{a_0}$, where $(a_2,a_1,a_0)\in X^{*}_+(\rT)$ is a restricted, generic dominant weight in the upper alcove, then an ordinary quasi-Breuil module as in (\ref{shape ordinary}) is endowed with a monodromy operator if and only if $[a_2-a_0]v_{2,0}=[a_1-a_0]v_{1,0}v_{2,1}$. This question is further explored in \cite{LLLM}.
\end{rmk}

\section{Potentially crystalline deformation rings}

In this section we explicitly compute certain potentially crystalline deformation rings with niveau 1 type. The main result is Theorem \ref{main}. We recall that $\rhobar: G_{\qp}\rightarrow \mathrm{GL}_3(\F)$ is ordinary, of the form
$$
\rhobar\vert_{I_{\qp}}\cong \maqn{lll}{{\omega^{a_2+2}}&{\ast_1}&{\ast}\\{0}&{\omega^{a_1+1}}&{\ast_2}\\{0}&{0}&{\omega^{a_0}}}
$$
verifying the genericity hypothesis $a_1-a_0,a_2-a_1>3$, $a_2-a_0<p-4$. As before we fix the principal series tame type $\tau\defeq \teich{\omega}^{a_2}\oplus\teich{\omega}^{a_1}\oplus\teich{\omega}^{a_0}$.

If $R$ is a complete local Noetherian $\cO_E$-algebra, recall from section \ref{integralp-adicHT} the following categories of semilinear algebra data:
\begin{equation*}
\xymatrix{
\RModdd[2]\ar[d]_{\otimes_{\cO_E}\F}\ar[r]&\RModddN[2]\ar^{\otimes_{\cO_E}\F}[d]\\
\RBrModdd[2]\ar[r]&\RBrModddN[2]
}
\end{equation*}
endowed with faithful, covariant functors $\mathrm{T}_{x}^{2}$ towards Galois (with $x\in\{\mathrm{st},\mathrm{qst}\}$), verifying the natural, evident compatibilities with respect to the forgetful, restriction and mod-$p$ reduction functors. 



\begin{defi}
\label{def gauge}
Let $\cM\in \OEModdd[2]$ be a strongly divisible lattice  of type $\tau$ such that $\Tst^{2}({\mathcal{M}})\otimes \F\cong \rhobar$. A basis $\underline{e}=(e_0,e_1,e_2)$ on $\cM$ is said to be a \emph{gauge basis} for $\cM$ if $\underline{e}$ is compatible with the descent data and there exists $\underline{f}=(f_0,f_1,f_2)\in\Fil^2\cM$ such that
\begin{eqnarray*}
\label{str div mod}
\mathrm{Mat}_{\underline{e}}([f_0,\,f_1,\,f_2])=\maqn{ccc}{{1}&{0}&{0}\\u^{\left[ a_1-a_0\right]}{x}&E(u)&0\\u^{\left[ a_2-a_0\right]}({y}'+E(u){y})&E(u)u^{\left[ a_2-a_1\right]}{z}&u^{2e}}
\end{eqnarray*}
and $\mathrm{Mat}_{\underline{e},\underline{f}}(\varphi_2)=\mathrm{Diag}({\alpha}_0,{\alpha}_1,{\alpha}_2),$ where ${x},y,y',z\in \Oe$ and ${\alpha}_i\in \Oe^{\times}$.

If $R$ is a complete local Noetherian $\Oe$-algebra we define in an analogous way the notion of \emph{gauge basis} for modules $\cM\in \RModdd[2],\,\barcM\in \RBrModdd[2]$ $($resp. $\cM\in \RModddN[2],\,\barcM\in \RBrModddN[2])$ of type $\tau$ such that $\Tst^{2}(\cM)\otimes _{R}\F\cong \rhobar$ $($resp. $\Tqst^{2}(\cM)\otimes _{R}\F\cong \rhobar\vert_{G_{(\Qp)_{\infty}}}).$

A morphism of strongly divisible lattices with gauge basis $(\cM_1,\un{e}_1)\rightarrow (\cM_2,\un{e}_2)$ is defined as a morphism $\cM_1\rightarrow \cM_2$ in $\OEModdd[2]$ such that $e_{i,1}\mapsto e_{i,2}$ for $i\in\{0,1,2\}$. We have the analogous definition for a morphism of $R$-valued $($quasi$)$-strongly divisible lattices and $($quasi$)$-Breuil modules, where $R$ is a complete local Noetherian $\Oe$-algebra $($resp. $\F$-algebra$)$.
\end{defi}

We record the following:
\begin{lem}
\label{lemchange}
Let $\cM\in \OEModdd[2]$ be a strongly divisible lattice  of type $\tau$ such that $\Tst^{2}({\mathcal{M}})\otimes \F\cong \rhobar$. Then $\cM$ is endowed with a gauge basis $\underline{e}$. Moreover if $\underline{e},\,\underline{e}'$ are two gauge bases on $\cM$, then there exists $t\in \bT_3(\Oe)$ such that $\underline{e}=\underline{e}'\cdot t$.

If $R$ is a complete local noetherian $\Oe$-algebra we have the evident, analogous statement for Breuil modules and quasi-Breuil modules with $R$-coefficients.
\end{lem}
\begin{proof}
The fact that $\cM$ is endowed with a gauge basis is immediate from Theorem \ref{phifilr}. Moreover, given two gauge basis $\underline{{e}},\,\underline{{e}}'$ for ${\mathcal{M}}$, one has 
$$
\widehat{e}'_{\bullet}=s_{\bullet}\widehat{e}_{\bullet}
$$ 
for $\bullet\in\{0,1,2\}$)  and $s_{\bullet}\in (S)_{\teich{\omega}^0}^{\times}$ (the latter denoting the invertible element in the $\teich{\omega}^0$-isotypical component of the Breuil ring $S$).

As $\underline{\widehat{e}}'$, $\underline{\widehat{e}}$ are gauge basis, the effect of the change of basis on the Frobenius action gives
$$
\alpha_{\bullet}'=\frac{\varphi(s_{\bullet})}{s_{\bullet}}\alpha_{\bullet}
$$
for $\bullet\in\{0,b,a\}$. As $x_{\bullet}\defeq\alpha_{\bullet}'\alpha_{\bullet}^{-1}\in \Oe^{\times}$ one deduces that $s_{\bullet}\in S^{\varphi=x_{\bullet}}=\Oe$.

The statement for Breuil modules and quasi-Breuil modules with $\Rbar$ coefficients is deduced following the analogous argument, using Proposition \ref{diagonalize mod p}.
\end{proof}





In what follows, we need the unicity of Breuil modules $\barcM$ verifying $\Tst^2(\barcM)=\rhobar$. The following proposition shows that this is the case when the descent datum on $\barcM$ is sufficiently generic with respect to the inertial weights of $\rhobar^{\mathrm{ss}}$. 
\begin{prop}
\label{unicity1}
Let $\barcM_1,\barcM_2\in\FBrModdd[2]$ be Breuil modules with descent data of type $\tau$.
Assume that $\Tst^{2}(\barcM_i)\cong \rhobar$ for $i\in\{1,2\}$, where $\rhobar$ is ordinary Fontaine-Laffaille and strongly generic as in (\ref{strong genericity}).

Then we have an isomorphism of Breuil modules $\barcM_1\stackrel{\sim}{\longrightarrow}\barcM_2$.
\end{prop}
\begin{proof}
Let $\barcM\in\FBrModdd[2]$ be a Breuil module with descent data of type $\tau$. 

By Proposition \ref{essential image} we have a gauge basis $\un{e}$ and a system of generators $\un{f}$ for $\Fil^2\barcM$ such that:
\begin{eqnarray*}
\mathrm{Mat}_{\underline{e}}(\mathrm{Fil}^2\barcM_0)=\small\maqn{ccc}{{1}&{0}&{0}\\u^{\left[ a_1-a_0\right]}x&u^e&0\\u^{e}u^{\left[ a_2-a_0\right]}y&u^eu^{\left[ a_2-a_1\right]}z&u^{2e}},\,\mathrm{Mat}_{\underline{e},\underline{f}}(\varphi_2)=\mathrm{Diag}(\alpha_0,\alpha_1,\alpha_2)
\end{eqnarray*}
for some $x,y,z\in \F$, $\alpha_i\in \F^{\times}$.
By Lemma \ref{lemma lawdd 1} the $(\F(\!(\underline{\pi})\!),\phi)$-module $\mathfrak{M}\defeq M_{\Fp(\!(\un{\pi})\!)}(\barcM^{\ast})$ is described by
$$
\mathrm{Mat}_{\underline{\mathfrak{e}}}(\phi)=
\maqn{ccc}{{1}&\underline{\pi}^{\left[ a_1-a_0\right]}x&\underline{\pi}^{e+\left[ a_2-a_0\right]}y\\
0&\underline{\pi}^e&\underline{\pi}^{e+\left[ a_2-a_1\right]}z\\0&0&\underline{\pi}^{2e}}\mathrm{Diag}(\alpha_0^{-1},\alpha_1^{-1},\alpha_2^{-1}).
$$

By considering the change of basis $\underline{\mathfrak{e}}'\defeq \left(\underline{\pi}^{a_0}\mathfrak{e}_0,\underline{\pi}^{a_1}\mathfrak{e}_1,\underline{\pi}^{a_2}\mathfrak{e}_2\right)$ (i.e., by considering the $\omega^{a_0}$-isotypical component of $\mathfrak{M}$), we obtain
$$
\mathrm{Mat}_{\underline{\mathfrak{e}'}}(\phi)=
\maqn{ccc}{{\underline{\pi}^{a_0}}&\underline{\pi}^{e(a_1+1)}x&\underline{\pi}^{e(a_2+2)}y\\
0&\underline{\pi}^{e(a_1+1)}&\underline{\pi}^{e(a_2+2)}z\\0&0&\underline{\pi}^{e(a_2+2)}}\mathrm{Diag}(\alpha_0^{-1},\alpha_1^{-1},\alpha_2^{-1}).
$$ 
which shows that $\mathfrak{M}$ is the base change to $\F(\!(\underline{\pi})\!)$ of the $(\F(\!(\underline{p})\!),\phi)$-module $\mathfrak{M}_0$ defined by
$$
\mathrm{Mat}_{\underline{\mathfrak{f}}}(\phi)=
\maqn{ccc}{\alpha_0^{-1}&\alpha_1^{-1}x&\alpha_2^{-1}y\\
0&\alpha_1^{-1}&\alpha_2^{-1}z\\0&0&\alpha_2^{-1}}\mathrm{Diag}(\underline{p}^{a_0},\underline{p}^{a_1+1},\underline{p}^{a_2+2}).
$$ 
for an appropriate basis $\underline{\mathfrak{f}}$ on $\mathfrak{M}_0$. Hence, the $(\F(\!(\underline{p})\!),\phi)$-module $\mathfrak{M}_0(-a_0)$ defined by
$$
\mathrm{Mat}_{\underline{\mathfrak{f}}}(\phi)=
\maqn{ccc}{\alpha_0^{-1}&\alpha_1^{-1}x&\alpha_2^{-1}y\\
0&\alpha_1^{-1}&\alpha_2^{-1}z\\0&0&\alpha_2^{-1}}\mathrm{Diag}(1,\underline{p}^{a_1-a_0+1},\underline{p}^{a_2-a_0+2})
$$ 
verifies $\Hom(\mathfrak{M}_0(-a_0),\Fp(\!(\underline{p})\!)^s)\cong \rhobar\otimes \omega^{-a_0}\vert G_{(\qp)_{\infty}}$.

By an evident change of basis and Proposition \ref{proposition comparison} we deduce that 
$\mathfrak{M}_0(-a_0)=\mathcal{F}(M)$ where $M$ is the Fontaine-Laffaille module in Hodge-Tate weights $(0,a_1-a_0+1,a_2-a_0+2)$ and whose Frobenii are described, in an appropriate basis, by
$$
\mathrm{Mat}(\phi_{\bullet})=\maqn{ccc}{1&x&y\\
0&1&z\\0&0&1}\mathrm{Diag}(\alpha_0^{-1},\alpha_1^{-1},\alpha_2^{-1}).
$$

We now specialize to our situation: for $i\in\{1,2\}$ we have:
\begin{eqnarray*}
\mathrm{Mat}_{\underline{e}}(\mathrm{Fil}^2\barcM_i)=\small\maqn{ccc}{{1}&{0}&{0}\\u^{\left[ a_1-a_0\right]}x_i&u^e&0\\u^{e}u^{\left[ a_2-a_0\right]}y_i&u^eu^{\left[ a_2-a_1\right]}z_i&u^{2e}},\,\mathrm{Mat}_{\underline{e},\underline{f}}(\varphi_2)=\mathrm{Diag}(\alpha_0,\alpha_1,\alpha_2)
\end{eqnarray*}
for some $x_i,y_i,z_i\in \F$ (and the $\alpha_i\in \F^{\times}$ uniquely determined by $\rhobar_i(\Frob_p)$) we deduce the Fontaine-Laffaille modules $M_i$, in Hodge-Tate weights $(0,a_1-a_0+1,a_2-a_0+2)$ Frobenii
$$
\mathrm{Mat}(\phi_{\bullet})=\maqn{ccc}{1&x_i&y_i\\
0&1&z_i\\0&0&1}\mathrm{Diag}(\alpha_0^{-1},\alpha_1^{-1},\alpha_2^{-1}).
$$

As have $\Tst^2(\barcM_1)\cong\Tst^2(\barcM_2)$ by assumption we deduce that the Fontaine-Laffaille modules above are isomorphic. By \cite{HM}, Lemma 2.1.7, any change of basis on $M_i$ which is compatible with both the Hodge and the submodule filtration on $M_i$ is diagonal; in other words, one has
$$
\maqn{ccc}{1&x_1&y_1\\
0&1&z_1\\0&0&1}=\maqn{ccc}{1&x_2&y_2\\
0&1&z_2\\0&0&1}\mathrm{Diag}(\lambda,\mu,\nu)
$$
for some $\lambda,\mu,\nu\in \F^{\times}$. We deduce an isomorphism of Breuil modules with gauge basis $\cM_1\stackrel{\sim}{\longrightarrow}\cM_2$ defined by $\un{e}_1\mapsto\un{e}_2\cdot\mathrm{Diag}(\lambda,\mu,\nu)$.
\end{proof}

We fix a pair $(\barcM,\underline{e})$ where $\barcM\in\FBrModdd[2]$ has type $\tau$ and verifies $\Tst^2(\barcM)\cong\rhobar$ and $\underline{e}$ is a  gauge basis on it. The basic setup will be similar to \S 7.4 of \cite{EGS}. We now introduce the following deformation rings. 

\begin{itemize}
	\item[$1)$] $R^{\Box}_{\rhobar}\defeq R^{\Box, (0,1,2), \mathrm{cris}, \tau}_{\rhobar}$ is the framed potentially crystalline deformation ring of $\rhobar$, with Galois type $\tau$ and $p$-adic Hodge type $(0,1,2)$;
	\item[$2)$] $R^{\tau}_{\barcM, \infty}$  represents the deformation functor of pairs $(\cM,\underline{\widehat{e}})$ where $\cM$ is a quasi-strongly divisible module lifting $\barcM$ and $\underline{\widehat{e}}$ is a gauge basis on $\cM$ reducing to the gauge basis $\underline{e}$ on $\barcM$.
	\item[$3)$] $R^{\tau}_{\barcM}$  represents the deformation functor of pairs $(\cM,\underline{\widehat{e}})$ as in $2)$ where $\cM$ is now a strongly divisible lattice.
	\item[$4)$] $R^{\Box,\tau}_{\barcM}$ represents deformation functor of triples $(\cM,\underline{\widehat{e}},\rho)$ where the pair $(\cM,\widehat{\un{e}})$ is as in $3)$ and $\rho\cong \Tst^2(\cM)$ (i.e., the pair $(\cM,\underline{\widehat{e}})$ comes with a framing on $\Tst^2(\cM)$);
	\item[$5)$] $R^{\Box}_{\barcM}$ parameterizing pairs $(\cM,\rho)$ where $\cM$ is as in $3)$ and $\rho\cong \Tst^2(\cM)$ (i.e. we fix a framing on $\Tst^2(\cM)$).
\end{itemize}

The relationship between the various deformation rings is summarized in the following diagram:
\begin{equation} \label{bigdiagram}
\xymatrix{
&  & \Spf R^{\Box,\tau}_{\barcM} \ar[dl]_{f.s.} \ar[dr]^{f.s.} & & \\
\Spf R^{\tau}_{\barcM, \infty} & \Spf R^{\tau}_{\barcM}  \ar@{_{(}->}[l]  &   & \Spf R^{\Box}_{\barcM} \ar[r]^{\sim}  & \Spf R^{\Box}_{\rhobar}
}
\end{equation}

Morphisms labelled f.s. are easily seen to be formally smooth.  Since the existence of monodromy is a closed condition, the leftmost arrow is a closed immersion. We will show that the rightmost arrow is an isomorphism in Theorem \ref{isomorphism moduli spaces} below. 

We now deduce the two important consequences of our work in the previous sections:
\begin{lem}
\label{closed}
The rings $R^{\tau}_{\barcM}\otimes_{\Oe} \F$, $R^{\tau}_{\barcM,\infty}\otimes_{\Oe}\F$ are formally smooth of dimension 6 and $7$ respectively.
Moreover we have a closed immersion 
$$
\Spf(R^{\tau}_{\barcM})\times \F \into \Spf(R^{\tau}_{\barcM,\infty})\times \F
$$
which realizes $\Spf(R^{\tau}_{\cM})\times \F$ as a linear subspace in $\Spf(R^{\tau}_{\cM,\infty})\times \F$
\end{lem}
\begin{proof}
By Theorem \ref{phifilr}, Proposition \ref{diagonalize mod p} and Lemma \ref{lemchange} we see that $\Spf(R^{\tau}_{\cM,\infty})$ is formally smooth, of relative dimension $7$ over $\cO_E$ with a universal family of ``ordinary'' quasi-Breuil modules.  

By Proposition \ref{essential image}, it follows immediately that $R^{\tau}_{\barc{M}}\otimes_{\Oe} \F$ is formally smooth, of relative dimension $6$ and the closed immersion in the statement of the Lemma is defined by $X_7=0$, if $X_1,\dots, X_7$ is a system of local coordinates on $\Spf(R^{\tau}_{\barcM,\infty})\times \F$.
\end{proof}

\begin{theo}
\label{isomorphism moduli spaces}
The natural map $\Spf R^{\Box}_{\mathcal{M}} \ra \Spf R^{\Box}_{\rhobar}$ is an isomorphism.
\end{theo}
\begin{proof}
By Theorem \ref{liu}, Lemma \ref{lemchange} and the uniqueness of Proposition \ref{unicity1} we deduce that the forgetful functor induces an isomorphism on generic fibers
$$
\Spf(R^{\Box}_{\barcM})\times\Qp\stackrel{\sim}{\longrightarrow}\Spf(R^{\Box}_{\rhobar})\times\Qp.
$$ 

As  $R^{\Box}_{\rhobar}$ is flat over $\Oe$, it suffices to show that 
$\Spf(R^{\Box}_{\barcM}){\longrightarrow}\Spf(R^{\Box}_{\rhobar})$ is a closed immersion, i.e. that the induced map 
\begin{equation}
\label{reduced tangent}
\Hom(R^{\Box}_{\barcM},\F[\varepsilon])\rightarrow \Hom(R^{\Box}_{\rhobar},\F[\varepsilon])
\end{equation}
between the reduced tangent spaces is injective.

This can be done by Galois descent via Proposition \ref{proposition comparison}, following the technique of the proof of Proposition \ref{unicity1}. 
More precisely, let $(\barc{N},\rho)$ be a point in $\Hom(R^{\Box}_{\barcM},\F[\varepsilon])$. By Lemma \ref{lemchange} $\barc{N}$ is endowed with a gauge basis $\underline{e}=(e_0,e_1,e_2)$; in particular we have
\begin{equation}
\mathrm{Mat}_{\underline{e}}([f_0,\,f_1,\,f_2])=\maqn{ccc}{{1}&{0}&{0}\\u^{\left[ a_1-a_0\right]}x&u^e&0\\u^{e}u^{\left[ a_2-a_0\right]}y&u^eu^{\left[ a_2-a_1\right]}z&u^{2e}},\qquad\mathrm{Mat}_{\underline{e},\underline{f}}(\varphi_2)=\mathrm{Diag}(\alpha_0,\alpha_1,\alpha_2)
\end{equation}
where $\underline{f}=(f_0,f_1,f_2)$ is a system of generators for $\Fil^2\barc{N}$ and $x,y,x\in \F[\varepsilon]$, $\alpha_i\in \F[\varepsilon]^{\times}$.

As in the proof of Proposition \ref{unicity1}, we deduce that
$$
\rho\cong \Tcris^*(N)\otimes \omega^{a_0}
$$
where $N$ is a Fontaine-Laffaille module over $\F[\varepsilon]$, in Hodge-Tate weights $(0,a_1-a_0+1,a_2-a_0+2)$ and Frobenii given by:
$$
\mathrm{Mat}(\phi_{\bullet})=\maqn{ccc}{1&x&y\\
0&1&z\\0&0&1}\mathrm{Diag}(\alpha_0^{-1},\alpha_1^{-1},\alpha_2^{-1}).
$$

Hence, the image of the reduced tangent map (\ref{reduced tangent}) consists of Fontaine-Laffaille tangent vectors, and the map is therefore injective.
\end{proof}

The following elementary result is the reason we only needed to compute the monodromy on the special fiber:
\begin{prop}
\label{formalsmoothness}
Let $R$ be a complete local Noetherian $\cO_E$-algebra with residue field $\F$.  If $R\otimes_{\cO_E} \F$ is formally smooth of dimension $d$ and $R$ is flat over $\Oe$, then $R$ is formally smooth over $\Oe$ of dimension $d$.
\end{prop}
\begin{proof}
As the special fiber $R\otimes_{\Oe} \F$ is formally smooth over $\F$, we have an isomorphism $R\otimes_{\Oe} \F = \F[\![x_1, \ldots, x_d]\!]$. Choose any lifts $\widetilde{x}_i \in m_R$ where $m_R$ is the maximal ideal of $R$.  By Nakayama, we have a surjective map 
\begin{eqnarray*}
\Oe[\![X_1,\dots,X_n]\!] \ra R
\end{eqnarray*}
which is an isomorphism on the special fiber.  As $R$ is $\Oe$-flat, the above map is in fact an isomorphism.  
\end{proof}

We are now ready to prove our main result:
\begin{theo} \label{main} Let $R^{\Box, (0,1,2), \mathrm{cris}, \tau}_{\rhobar}$ be the framed potentially crystalline deformation ring for $\rhobar$, with Galois type $\tau$ and $p$-adic Hodge type $(0,1,2)$.  Assume that $\Spf R^{\Box, (0,1,2), \mathrm{cris}, \tau}_{\rhobar}$ is non-empty.  Then  $R^{\Box, (0,1,2), \mathrm{cris}, \tau}_{\rhobar}$ is formally smooth over $\Oe$ of relative dimension 12.  
\end{theo}  
\begin{proof}[Proof of Theorem \ref{main}]  
By Theorem \ref{isomorphism moduli spaces}, we have  $\Spf R^{\Box}_{\mathcal{M}}$ is non-empty and flat over $\Oe$.  Thus, the same is true for $\Spf R^{\Box, \tau}_{\barcM}$ and $\Spf R^{\tau}_{\barcM}$.   Combining Lemma \ref{closed} and Proposition \ref{formalsmoothness}, we deduce that  $R^{\tau}_{\barcM}$ is formally smooth over $\Oe$.  

From diagram (\ref{bigdiagram}), we deduce that $\Spf R^{\Box}_{\mathcal{M}}$ is formally smooth over $\Oe$. The relative dimension follows from \cite[Theorem 3.3.8]{PST}  given that $R^{\Box}_{\rhobar}$ is formally smooth over $\Oe$.  
\end{proof}  

\begin{rmk}  The condition that $\Spf R^{\Box, (0,1,2), \mathrm{cris}, \tau}_{\rhobar}$  is non-empty arises because we only compute the monodromy on the special fiber.  As a result, we do not exhibit any potentially crystalline lifts of $\rhobar$. 
This problem is addressed in \cite{LLLM}. In personal correspondence, Hui Gao informed us that he was able independently to prove that the ring $\Spf R^{\Box, (0,1,2), \mathrm{cris}, \tau}_{\rhobar}$ in Theorem \ref{main} is non-empty $($\cite{Gao}$)$.
\end{rmk} 

We can also deduce the following:
\begin{cor} \label{formal smoothness} The ring $R^{\tau}_{\barcM,\infty}$ is formally smooth over $\Oe$ of relative dimension $7$. Assume that $R^{\tau}_{\barcM} \otimes_{\Oe} E$ is non-empty, then $\Spf R^{\tau}_{\barcM}$ is formally smooth over $\Oe$ of dimension $6$ respectively.
\end{cor}
\begin{proof}
It follows immediately from Proposition \ref{formalsmoothness} and Lemma \ref{closed} if the rings are $\cO_E$-flat.  The ring $R^{\tau}_{\barcM,\infty}$ is flat since we produced families of quasi-strongly divisible modules lifting $\barcM$ in  \S 3.  If $R^{\tau}_{\barcM} \otimes_{\Oe} E$ is non-empty then so is $R^{\Box}_{\rhobar} \otimes_{\Oe} E$ and so we can argue as in the proof of Theorem \ref{main}.
\end{proof}

\paragraph{\textbf{Acknowledgements}}
The authors wish to express their deepest gratitude to Florian Herzig, whose insight, suggestions and remarks have been crucial to this work. 

For several enlightening discussions while conceiving this work the second author would like to heartily thank Christophe Breuil, Ariane M\'ezard and Benjamin Schraen.

\bibliography{Biblio}
\bibliographystyle{amsalpha}
\end{document}